\documentclass[11pt]{article}

\usepackage[a4paper,margin=1in]{geometry}
\usepackage[T1]{fontenc}
\usepackage[utf8]{inputenc}
\usepackage{lmodern}
\usepackage{microtype}
\usepackage{amsmath,amssymb,amsthm,mathtools,mathrsfs,esint}
\usepackage{enumitem}
\usepackage[colorlinks=true,citecolor=blue,linkcolor=blue,urlcolor=blue]{hyperref}
\usepackage{indentfirst}
\allowdisplaybreaks
\setlist[itemize]{topsep=3pt,itemsep=1pt,parsep=0pt}

\newtheorem{theorem}{Theorem}[section]
\newtheorem{proposition}[theorem]{Proposition}
\newtheorem{lemma}[theorem]{Lemma}
\newtheorem{corollary}[theorem]{Corollary}
\newtheorem{definition}[theorem]{Definition}

\newcommand{\RCD}{\mathrm{RCD}}
\newcommand{\CAT}{\mathrm{CAT}}
\newcommand{\dist}{\mathsf d}
\newcommand{\meas}{\mathfrak m}
\newcommand{\setR}{\mathbb R}
\newcommand{\di}{\,\mathrm d}
\newcommand{\Lip}{\operatorname{Lip}}
\newcommand{\Lipbs}{\operatorname{Lip}_{\mathrm{bs}}}
\newcommand{\lip}{\operatorname{lip}}
\newcommand{\loc}{\mathrm{loc}}
\newcommand{\ks}{\mathsf{KS}}
\newcommand{\kse}{\mathsf{ks}}
\newcommand{\Leb}{\mathscr L}
\newcommand{\eps}{\varepsilon}
\newcommand{\supp}{\operatorname{supp}}
\newcommand{\norm}[1]{\left\lVert #1\right\rVert}

\title{Lipschitz regularity of Sobolev maps from
$\RCD(K,N)$ spaces to $\CAT(0)$ spaces}
\author{Bang-Xian Han \and Zeyu Wang}
\date{October 29, 2026}
\begin{document}

\maketitle

\begin{abstract}
We establish a quantitative criterion for local Lipschitz
regularity of Sobolev maps in the sense of Korevaar and Schoen from finite-dimensional \(\RCD(K,N)\) spaces into complete \(\CAT(0)\) spaces, without assuming harmonicity. We quantify the failure of local energy minimality by the
harmonic replacement potential, a dyadic sum of the square roots
of normalized energy gaps between a map and its harmonic
replacements. We prove that this potential, together with an
averaged energy term, controls the minimal weak upper gradient
pointwise, with constants depending only on the
curvature-dimension bounds and the scale. In particular, local boundedness of the replacement potential implies that the map admits a locally Lipschitz representative.
\end{abstract}

\noindent\textbf{Keywords.}
Harmonic maps; harmonic replacement; $\RCD$ spaces; $\CAT(0)$ targets;
Lipschitz regularity; Green functions.

\smallskip
\noindent\textbf{2020 Mathematics Subject Classification.}
58E20, 53C23, 30L99, 46E35.

\tableofcontents

\section{Introduction\label{sec:introduction}}
\subsection{Background and motivation }
Harmonic maps between Riemannian manifolds are the critical points of the Dirichlet energy. The study of harmonic maps is closely related to minimal surface theory and to Bochner's method; see \cite{Morrey1948,SchoenYau1997}. For nonpositively curved targets, the Dirichlet energy is convex along geodesic homotopies. For the classical theory of harmonic maps into nonpositively curved manifolds, we refer to Eells--Sampson \cite{EellsSampson64}, and, for manifolds with boundary, Hamilton \cite{Hamilton1975}. For general smooth targets, the
interior and boundary regularity theory was developed by
Schoen--Uhlenbeck \cite{SchoenUhlenbeck1982,SchoenUhlenbeck1983}; see
also \cite{HildebrandtKaulWidman1977,JostMeier1983,GiaquintaGiusti1982,
GiaquintaGiusti1984} for related variational results.

The extension to singular spaces requires an energy defined without
coordinates on the target. In the remarkable work of Gromov and Schoen \cite{GromovSchoen}, the authors proposed a variational approach for the theory of harmonic mappings to the setting of mappings into singular metric spaces, along with important applications to rigidity problems for certain discrete groups. Korevaar--Schoen \cite{KorevaarSchoen}
developed an intrinsic Sobolev and energy theory for maps from
Riemannian domains into complete metric spaces and established
existence and uniqueness for the Dirichlet problem with
nonpositively curved targets. Independently, Jost
\cite{Jost94} used averages of squared distances, equilibrium maps,
and $\Gamma$-convergence to develop a variational approach already
allowing metric source spaces. He subsequently developed the theory
through convex functionals and generalized Dirichlet forms
\cite{Jost95,Jost97}. Sturm \cite{Sturm2001,Sturm2005}
pursued a parallel route through Dirichlet forms and Markov semigroups.
In this paper, harmonic maps are understood to be local minimizers
of the Korevaar--Schoen energy.

The central regularity question is whether the local Lipschitz
continuity proved by Korevaar--Schoen \cite{KorevaarSchoen} for smooth
Riemannian sources persists on singular domains. Lin \cite{Lin1997}
obtained H\"older regularity results for Alexandrov sources with
curvature bounded below, while Jost \cite{Jost97} established related
estimates in the framework of generalized Dirichlet forms.
Lin conjectured local
Lipschitz continuity in the Alexandrov setting, and Jost posed a
related question in a more general framework
\cite[p.~38]{Jost1998}. Zhang--Zhu \cite{ZhangZhu2018} resolved
Lin's conjecture by proving local Lipschitz continuity for harmonic
maps from Alexandrov sources with curvature bounded below into
NPC spaces.

Finite-dimensional $\RCD(K,N)$ spaces \cite{AGSRCD14,EKS15}
provide a framework for studying this regularity problem under
a synthetic lower Ricci curvature bound. Gigli--Tyulenev
\cite{GigliTulyenev21} developed the Korevaar--Schoen energy
and Dirichlet theory in this setting, with complete $\CAT(0)$
targets. Their results on strongly rectifiable spaces apply
also to collapsed $\RCD(K,N)$ spaces and give a representation
of the energy density in terms of the metric differential.

The regularity question is whether a lower Ricci curvature
bound suffices for quantitative interior Lipschitz estimates.
For smooth Riemannian sources and complete $\CAT(0)$ targets,
Zhang--Zhong--Zhu \cite{ZhangZhongZhu2019} established such
estimates using a distributional Bochner inequality for the
squared pointwise Lipschitz constant. In the $\RCD$ setting,
local Lipschitz continuity, with quantitative interior estimates,
was established in independent works by
Gigli \cite{GigliLipschitz} and Mondino--Semola
\cite{MondinoSemola2026}. Gigli proved a version of the
Zhang--Zhong--Zhu inequality in this setting.
Mondino--Semola also established a Bochner inequality for the
squared pointwise Lipschitz constant containing the nonnegative
term $|D(\operatorname{lip}u)|^2$. Boundary regularity for
such maps was studied by Zhang--Zhu
\cite{ZhangZhuBoundary2024,ZhangZhuOptimal2025}.

These Bochner inequalities concern the squared pointwise
Lipschitz constant, which, for smooth maps between Riemannian
manifolds, equals the squared operator norm of the differential.
The squared energy density corresponds instead, up to normalization,
to the squared Hilbert--Schmidt norm. This distinction is
important for the comparison argument below, where harmonic
replacement energy gaps control the $L^2$-difference between
the minimal weak upper gradients of a map and its harmonic
replacement.

Our question concerns maps that need not locally minimize the
Dirichlet energy: \emph{How much may a Sobolev map deviate from local Dirichlet minimality while remaining locally Lipschitz?}
Such maps arise naturally in variational models of phase
transitions and free boundaries, where a Dirichlet term is
combined with a potential energy. Beyond these
variational applications, our original motivation for this question
came from a different direction, the local Lipschitz regularity of
the time discretization of the harmonic map heat flow, whose
comparison structure leads to the same potential.

The first step is to identify a suitable quantity measuring
the failure of local energy minimality. On a ball $B$, we
compare $u$ with its unique harmonic replacement $h_B$, which minimizes
the Dirichlet energy with the same boundary values. The resulting nonnegative energy gap \(E(u;B)-E(h_B;B)\)
is defined for every finite-energy map, without reference to
a target-valued equation or a tension field.

To pass from this local energy gap to a pointwise estimate, we
adapt the comparison method of nonlinear potential theory. In
the elliptic setting, one compares a solution $u$ of an
inhomogeneous equation $Lu=f$ on a ball $B$ with its harmonic
replacement, the solution of the homogeneous equation $Lu=0$ on
$B$ with the same boundary values, and the comparison errors,
accumulated over dyadic balls, are controlled by a potential of
the data $f$ (of Riesz or Wolff type), which yields the
pointwise estimates for $u$; see \cite{KuusiMingioneGuide14,DuzaarMingione10}.
In our setting no equation is available; the
harmonic replacement $h_B$ still plays the role of the homogeneous
solution, and the energy gap $E(u;B)-E(h_B;B)$ plays the role of
the data. Accumulating these gaps over dyadic balls defines the
\emph{harmonic replacement potential}; see
Definition~\ref{def:replacement-potential} below.

\subsection{Main result}
Let $(X,\dist,\meas)$ be an $\RCD(K,N)$ space with full support,
where $K\in\setR$ and $1<N<\infty$, and let $(Y,\dist_Y)$ be a complete
$\CAT(0)$ space. Fix a bounded open set $\Omega\subset X$,
a point $o\in X$, and $R>0$ such that
\begin{equation}
 B_{8R}(o)\Subset\Omega,
 \qquad X\setminus B_{16R}(o)\neq\varnothing.
 \label{eq:intro-local-setting}
\end{equation}
A ball $B=B_r(x)$ is called \emph{admissible} if $0<r\le R$ and
$B_r(x)\subset B_{2R}(o)$.

We write $\ks^{1,2}(\Omega,Y)$ for the space of
$L^2$ maps with finite Korevaar--Schoen energy;
see Definition~\ref{def:ks-space}. This finite-energy assumption on a
fixed domain is stronger than requiring $L^2$-integrability and finite
energy only on relatively compact open subsets. All estimates below
are interior. If a map satisfies these local conditions on a larger
open set $\Omega_{\mathrm{old}}$ with
$B_{8R}(o)\Subset\Omega_{\mathrm{old}}$, properness of $X$ allows us to
choose a bounded open neighborhood $\Omega_0$ of
$\overline{B_{8R}(o)}$ such that $\Omega_0\Subset\Omega_{\mathrm{old}}$.
The results apply to $u|_{\Omega_0}\in\ks^{1,2}(\Omega_0,Y)$.
Restriction leaves the comparison balls, harmonic replacements,
energy gaps, and replacement potentials unchanged.

Let $n$ denote
the essential dimension of $X$
\cite{MondinoNaber19,BrueSemola20}. For
$u\in\ks^{1,2}(\Omega,Y)$, write $\mathrm du$ for the differential
of a metric-valued Sobolev map in the sense of
\cite{GigliPasqualettoSoultanis2020}. Throughout, $|\mathrm du|$
denotes its pointwise operator norm, equivalently the minimal
$2$-weak upper gradient, whereas $|\mathrm du|_{\mathrm{HS}}$
denotes its Hilbert--Schmidt norm. We use the density and energy
normalization
\begin{equation}
 \begin{aligned}
 e_2[u]&:=(n+2)^{-1/2}|\mathrm du|_{\mathrm{HS}},\\
 E(u;U)&:=\frac{n+2}{2}\int_U e_2[u]^2\di\meas
          =\frac{n+2}{2}\mathrm E_2^U(u).
 \end{aligned}
 \label{eq:intro-energy-normalization}
\end{equation}
We use the energy density notation of Mondino--Semola
\cite{MondinoSemola2026}; $\mathrm E_2^U$ denotes the
Korevaar--Schoen energy of Gigli--Tyulenev
\cite[Theorem~5.7 and Proposition~6.7]{GigliTulyenev21}. In particular,
\begin{equation}
 |\mathrm du|^2\le(n+2)e_2[u]^2\le n |\mathrm du|^2
 \qquad\meas\text{-a.e.}
 \label{eq:intro-trace-op}
\end{equation}
For smooth Euclidean maps, these are $|\mathrm du|=\|Du\|_{\mathrm{op}}$ and
$e_2[u]=(n+2)^{-1/2}\|Du\|_{\mathrm{HS}}$; for scalar maps,
$|\mathrm du|=\sqrt{n+2}\,e_2[u]=|Du|$. Thus $E$ agrees with the usual scalar Dirichlet
energy. The representation by quadratic forms, used in the comparison
argument, is recorded in Theorem~\ref{thm:irc-sobolev-energy}.

For an admissible ball $B$ and $u\in\ks^{1,2}(B,Y)$, following \cite[Definition~5.8]{GigliTulyenev21}, let
\[
 \ks^{1,2}_u(B,Y)
 :=\{v\in\ks^{1,2}(B,Y):\dist_Y(u,v)\in W^{1,2}_0(B)\},
\]
and denote by
\[
 h_B:=\operatorname*{argmin}_{v\in\ks^{1,2}_u(B,Y)}E(v;B)
\]
the unique harmonic replacement of $u$ in $B$. Here the boundary condition is understood in the Sobolev sense; see \cite[Theorem~6.4]{GigliTulyenev21} and Proposition~\ref{prop:irc-exact}. 

\begin{definition}[Harmonic replacement potential]
\label{def:replacement-potential}
Let $u\in\ks^{1,2}(\Omega,Y)$. For every admissible ball
$B_r(y)$, define the replacement energy gap and the
normalized replacement defect by
\begin{equation}
 \begin{aligned}
 G_u(B_r(y))
 &:=E(u;B_r(y))-E(h_{B_r(y)};B_r(y)),\\
 \delta_u(B_r(y))
 &:=\left(\frac{G_u(B_r(y))}{\meas(B_r(y))}\right)^{1/2}.
 \end{aligned}
 \label{eq:intro-gap}
\end{equation}
For $x\in B_R(o)$, the harmonic replacement potential of $u$
at scale $R$ is
\begin{equation}
 \mathcal P_Ru(x)
 :=\sum_{j=0}^{\infty}
 \delta_u\bigl(B_{2^{-j}R}(x)\bigr).
 \label{eq:intro-potential}
\end{equation}
The series is allowed to take the value $+\infty$.
\end{definition}

We also use the averaged energy
\[
 \mathcal E_s(u;x)
 :=\left((n+2)\fint_{B_s(x)}e_2[u]^2\di\meas\right)^{1/2},
 \qquad B_s(x)\subset\Omega,
\]
and the essential oscillation
\[
 \operatorname*{ess\,osc}_{A}u
 :=\operatorname*{ess\,sup}_{(y,z)\in A\times A}
   \dist_Y(u(y),u(z)),
\]
where the essential supremum is taken with respect to
$(\meas|_A)\otimes(\meas|_A)$.

\begin{theorem}
\label{thm:replacement-potential}
Assume \eqref{eq:intro-local-setting} and let
$u\in\ks^{1,2}(\Omega,Y)$. There exists a constant
$C=C(N,K_-R^2)$, where $K_-:=\max\{-K,0\}$, such that, for
$\meas$-almost every $x\in B_R(o)$,
\begin{equation}
 |\mathrm du|(x)
 \leq C\bigl[\mathcal E_{2R}(u;o)+\mathcal P_Ru(x)\bigr].
 \label{eq:main-pointwise-estimate}
\end{equation}
In particular, if
\begin{equation}
 \sup_{x\in B_R(o)}\mathcal P_Ru(x)<\infty,
 \label{eq:main-uniform-potential}
\end{equation}
then $u$ admits a representative $\widetilde u$ on $B_{R/2}(o)$ satisfying
\begin{equation}
 \Lip\bigl(\widetilde u|_{B_{R/2}(o)}\bigr)
 \leq C\left[\mathcal E_{2R}(u;o)
 +\sup_{x\in B_R(o)}\mathcal P_Ru(x)\right].
 \label{eq:main-lipschitz-estimate}
\end{equation}
If, in addition,
\[
 \operatorname*{ess\,osc}_{B_{2R}(o)}u<\infty,
\]
then $\mathcal E_{2R}(u;o)$ in both
\eqref{eq:main-pointwise-estimate} and
\eqref{eq:main-lipschitz-estimate} may be replaced by
\[
 \frac{\operatorname*{ess\,osc}_{B_{2R}(o)}u}{R}.
\]

\end{theorem}

To the best of our knowledge, the pointwise estimate above, which
bounds the minimal weak upper gradient of a finite-energy map by
its averaged energy and the accumulated energy gaps of its
harmonic replacements, has not previously been stated in this form,
even for smooth Riemannian sources.
Theorem~\ref{thm:replacement-potential} establishes it for
arbitrary finite-energy maps from $\RCD(K,N)$ spaces into
complete $\CAT(0)$ targets, with constants depending only
on $N$ and $K_-R^2$.

Local boundedness of the replacement potential is sufficient, but not
necessary, for local Lipschitz regularity. For instance, let
$u(x)=|x|$ on $\Omega=B_{10R}(0)\subset\mathbb R^n$, with $n\ge2$.
For $0<r\le R$, its harmonic replacement on $B_r(0)$ is the constant
$r$. Thus $\delta_u(B_r(0))=1/\sqrt2$ and
$\mathcal P_Ru(0)=\infty$, although $u$ is $1$-Lipschitz.

\subsection{Strategy of the proof.}
On a fixed ball \(B\), we first show that the replacement energy gap \(G_u(B)\) controls the squared \(L^2\)-distance between \(|\mathrm du|\) and \(|\mathrm dh|\), where \(h=h_B\). Set
\[
J(x):=\int_{S^{n-1}}
|s_{u,x}(\xi)-s_{h,x}(\xi)|^2\,\di \sigma_n(\xi),
\]
where \(s_{u,x}\) and \(s_{h,x}\) are the metric differentials expressed in a common source frame, and \(\sigma_n\) is normalized surface measure. Then
\[
\int_B\bigl|\,|\mathrm du|-|\mathrm dh|\,\bigr|^2\,\di \mathfrak m
\le \frac{n(n+2)}2\int_B J\,\di \mathfrak m
\le 2(n+2)G_u(B).
\]
The first inequality follows by comparing the two Hilbertian seminorms through the fourth-moment identity \eqref{eq:irc-fourth-moment}; the second follows from the midpoint convexity and the minimality of \(h\). The full argument is given in the proof of Proposition~\ref{prop:irc-operator-gap}.

For a harmonic replacement \(h\), the Zhang--Zhong--Zhu inequality, 
recalled in Proposition~\ref{prop:irc-exact}, gives
\[
  \frac12\Delta(|\mathrm dh|^2) \ge K |\mathrm dh|^2\,\meas.
\]
Consequently, with
\[
  b(z)=\cosh\!\bigl(\sqrt{2K_-}\,z_1\bigr),
\]
the function \(b(z)|\mathrm dh|^2(y)\) is subharmonic on the corresponding
product domain. 

We next compare local averages of \(|\mathrm du|\) at successive scales.
For a fixed point \(x\), the averaging is taken against the
probability measures \(\sigma_j\) constructed in
Theorem~\ref{thm:gbl-averages}: they live on thick Green bands in
\(X\times\mathbb R^3\) and are built from Green-level mean-value
formulas\ \cite{Watson91,Ni07}.
Set
\[
  A_j:=
  \left(
    \int b(z)|\mathrm du|^2(y)\,\di\sigma_j(y,z)
  \right)^{1/2}.
\]
This gives
\[
\begin{aligned}
  A_{j+1}
  &\le A_j+2C\delta_u(\widehat B_j).
\end{aligned}
\]

Since the radii of \(\widehat B_j\) form a fixed dyadic
subsequence and \(A_j\to |\mathrm du|(x)\) at Lebesgue points of \(|\mathrm du|^2\),
iteration yields
\[
  |\mathrm du|(x)\le A_0+C\mathcal P_Ru(x).
\]
The projected-density bound gives
\(A_0\le C\mathcal E_{2R}(u;o)\), yielding the pointwise estimate~\eqref{eq:main-pointwise-estimate}. For the oscillation estimate, comparison with \(h_{B_R(x)}\)
and the interior gradient estimate give
\[
  A_0\le C\left[
    \frac{\operatorname*{ess\,osc}_{B_{2R}(o)}u}{R}
    +\delta_u(B_R(x))
  \right],
\]
where \(\delta_u(B_R(x))\le\mathcal P_Ru(x)\).

If \(M_R:=\sup_{x\in B_R(o)}\mathcal P_Ru(x)<\infty\),
the above estimate gives an \(L^\infty\) bound for \(|\mathrm du|\)
on \(B_R(o)\). By the local metric Sobolev-to-Lipschitz
property stated in Lemma~\ref{lem:irc-sobolev-lipschitz},
\(u\) therefore admits a representative \(\widetilde u\) satisfying
\[
  \Lip\bigl(\widetilde u|_{B_{R/2}(o)}\bigr)
  \le C\bigl[\mathcal E_{2R}(u;o)+M_R\bigr].
\]
Applying the same argument to the pointwise oscillation estimate
gives the corresponding Lipschitz bound with
\(\mathcal E_{2R}(u;o)\) replaced by
\(R^{-1}\operatorname*{ess\,osc}_{B_{2R}(o)}u\).
The full argument is given in Section~\ref{sec:proof-replacement}.

\section{Preliminaries}
\label{sec:preliminaries}
Throughout, $(X,\dist,\meas)$ and $(Y,\dist_Y)$ satisfy the assumptions stated in the Introduction. All metric-valued maps are Borel
and essentially separably valued, and maps that agree
$\meas$-almost everywhere are identified. We retain the admissible
balls and the energy normalization fixed there. Unless additional
parameters are indicated, constants depend only on $N$ and
$K_-R^2$, where $K_-:=\max\{-K,0\}$.

\subsection{Geometric analysis on RCD spaces}
\label{subsec:ms-background}
We recall the scalar Sobolev calculus developed in
\cite{AGS14,AGSRCD14}.
We denote by \(\Lip(X)\) and \(\Lipbs(X)\) the spaces of
real-valued Lipschitz functions on \(X\) and those with bounded
support, respectively. For \(f\in\Lip(X)\), its slope is
\[
  \lip f(x):=
  \limsup_{\substack{y\to x\\y\ne x}}
  \frac{|f(y)-f(x)|}{\dist(x,y)},
\]
with \(\lip f(x)=0\) at isolated points.
We write \(\Lip(f)\) for the Lipschitz constant of \(f\).

For \(f\in L^2(X,\meas)\), the Cheeger energy is defined by
\[
  \mathrm{Ch}(f):=
  \inf\left\{
    \liminf_{k\to\infty}\frac12\int_X(\lip f_k)^2\di\meas
    \;:\;
    f_k\in\Lipbs(X),\quad
    f_k\to f \text{ in } L^2(X,\meas)
  \right\}.
\]
The Sobolev space
\[
  W^{1,2}(X):=
  \{f\in L^2(X,\meas):\mathrm{Ch}(f)<\infty\}
\]
is equipped with the norm
\[
  \|f\|_{W^{1,2}(X)}^2
  :=\|f\|_{L^2(X,\meas)}^2+2\mathrm{Ch}(f).
\]

For \(f\in W^{1,2}(X)\), a nonnegative function
\(g\in L^2(X,\meas)\) is a relaxed slope of \(f\) if there
exist \(f_k\in\Lipbs(X)\) and \(\widetilde g\in L^2(X,\meas)\)
such that
\[
  f_k\to f,\qquad
  \lip f_k\rightharpoonup\widetilde g
  \quad\text{in }L^2(X,\meas),
  \qquad
  \widetilde g\le g\quad\meas\text{-a.e.}
\]
The set of relaxed slopes is nonempty, closed, and convex.
Its unique element of minimal \(L^2\)-norm is denoted by
\(|Df|\). It coincides with the minimal \(2\)-weak upper
gradient of \(f\), and
\[
  \mathrm{Ch}(f)=\frac12\int_X|Df|^2\di\meas.
\]
We shall also use the Sobolev-to-Lipschitz property:
if \(f\in W^{1,2}(X)\) satisfies \(|Df|\le L\)
\(\meas\)-almost everywhere, then \(f\) admits an
\(L\)-Lipschitz representative; the local form used below
is recorded in \cite[Proposition~2.10]{GigliViolo2023}.

For an open set $U\subset X$, let $\Lip_c(U)$ denote the
Lipschitz functions with compact support in $U$. Following
\cite[Section~2.1.3]{GigliViolo2023}, we say that
$f\in L^2_{\loc}(U)$ belongs to $W^{1,2}_{\loc}(U)$ if
$\chi f$ belongs to
$W^{1,2}(X)$ for every $\chi\in\Lip_c(U)$.
For such a function, the minimal weak upper gradient is
defined locally: on each open set $V\Subset U$, we set
$|Df|:=|D(\chi f)|$ almost everywhere on $V$, where
$\chi\in\Lip_c(U)$ equals one on $V$. By locality, this
definition is independent of the choice of $\chi$. We also define the spaces
\[
\begin{aligned}
 W^{1,2}(U)
 &:= \bigl\{f\in W^{1,2}_{\loc}(U)\mid f,|Df|\in L^2(U)\bigr\},\\
 W^{1,2}_0(U)
 &:= \overline{\Lip_c(U)}^{\,W^{1,2}(X)}\subset W^{1,2}(X).
\end{aligned}
\]
We shall use the following property of $W^{1,2}_0(U)$:
for $f\in W^{1,2}_0(U)$ and $g\in W^{1,2}(U)$,
\begin{equation}
 |g|\le f\quad\meas\text{-a.e. on }U
 \quad\Longrightarrow\quad
 g\in W^{1,2}_0(U).
 \label{eq:irc-wzero-domination}
\end{equation}
see \cite[Proposition~5.2]{GigliTulyenev21}.

A function $v\in W^{1,2}_{\loc}(U)$ admits a measure-valued
Laplacian on $U$ if there exists a signed Radon measure
$\Delta v$ satisfying
\[
 \int_U\phi\di\Delta v
 =-\int_U\langle\nabla v,\nabla\phi\rangle\di\meas
 \qquad\text{for every }\phi\in\Lip_c(U).
\]
Laplacian inequalities are understood by testing against
nonnegative functions in $\Lip_c(U)$.

By local doubling, the Lebesgue differentiation theorem holds:
for each $F\in L^1_{\loc}(X)$, $\meas$-almost every point is
a Lebesgue point of $F$;
see \cite[Theorem~14.15]{HajlaszKoskela00}.

Let $P_t=e^{t\Delta}$ denote the heat semigroup associated with
the Cheeger energy. It is symmetric, conservative, and
$L^\infty$-contractive \cite{AGS14,AGSRCD14}.
Its continuous kernel $p_t(p,q)$ satisfies the Gaussian bounds
of \cite[Theorems~1.1--1.2]{JLZ16}: for $0<t\le R^2$,
\begin{equation}
 \frac{c}{\meas(B_{\sqrt t}(p))}
       \exp\!\left(-\frac{C\dist(p,q)^2}{t}\right)
 \le p_t(p,q)\le
 \frac{C}{\meas(B_{\sqrt t}(p))}
       \exp\!\left(-\frac{\dist(p,q)^2}{Ct}\right).
 \label{eq:irc-gaussian}
\end{equation}
For any fixed $c_0>0$, these bounds also hold for
$0<t\le c_0R^2$, with constants depending additionally on $c_0$.
Moreover, for every $\phi\in C_c(X)$ and $p\in X$,
\[
 P_t\phi(p)\to\phi(p)
 \qquad\text{as }t\downarrow0,\ \text{or equivalently,}\qquad 
p_t(p,\cdot)\meas\rightharpoonup\delta_p.
\]
For the Green construction in Section~\ref{sec:gbl-averages}, equip
$Z:=X\times\setR^3$ with the Euclidean product metric and
$\nu:=\meas\otimes\Leb^3$. Tensorization
\cite[Theorem~3.23]{EKS15} gives $\RCD(-K_-,N+3)$.
The tensorization of the Cheeger energy and heat semigroup
\cite[Section~6]{AGSRCD14} gives the product kernel
\begin{equation}
 p_t^Z((x,0),(y,z))
 =p_t^X(x,y)(4\pi t)^{-3/2}\exp(-|z|^2/(4t)).
 \label{eq:irc-product-kernel}
\end{equation}
The corresponding scalar estimates hold on $Z$ with dimension bound $N+3$.
In particular, local Lipschitz regularity and the interior gradient
estimate for scalar harmonic functions follow from the scalar case of
Proposition~\ref{prop:irc-exact}, applied on interior product balls.

\begin{proposition}
\label{prop:irc-scalar}
There exist constants $c,C>0$, depending only on $N$ and
$K_-R^2$, such that
\begin{equation}
 \frac{\meas(B_s(x))}{\meas(B_r(x))}
 \ge c\left(\frac{s}{r}\right)^N
 \qquad\text{for every }x\in X,\quad 0<s\le r\le4R.
 \label{eq:irc-volume}
\end{equation}
Let $B_r=B_r(x)$ be an admissible ball and
$w\in W^{1,2}_0(B_r)$. Then
\begin{equation}
 \left(\fint_{B_r}|w|^2\di\meas\right)^{1/2}
 \le Cr
       \left(\fint_{B_r}|Dw|^2\di\meas\right)^{1/2}.
 \label{eq:irc-poincare}
\end{equation}
More generally, for $0<t<2N/(N-2)$ when $N>2$, and for
every finite $t>0$ when $1<N\le2$,
\begin{equation}
 \left(\fint_{B_r}|w|^t\di\meas\right)^{1/t}
 \le C_{N,t,K_-R^2}r
       \left(\fint_{B_r}|Dw|^2\di\meas\right)^{1/2}.
 \label{eq:irc-sobolev-moments}
\end{equation}
For $t<1$, the left-hand side is understood as a normalized
moment.
In particular, set
\begin{equation}
 q_*(N):=
 \begin{cases}
  2N/(N+2),&N>2,\\
  1,&1<N\le2.
 \end{cases}
 \label{eq:irc-qstar}
\end{equation}
If $q_*(N)<q<N$ and $q'=q/(q-1)$, then
\begin{equation}
 \left(\fint_{B_r}|w|^{q'}\di\meas\right)^{1/q'}
 \le C_{N,q,K_-R^2}r
       \left(\fint_{B_r}|Dw|^2\di\meas\right)^{1/2}.
 \label{eq:irc-sobolev}
\end{equation}
\end{proposition}

\begin{proof}
Bishop--Gromov comparison
\cite[Theorem~2.3 and Corollary~2.4]{Sturm06II}
gives \eqref{eq:irc-volume} and local doubling.
Since $X$ also satisfies the curvature-dimension condition
$\mathrm{CD}(-K_-,N)$, where $-K_-\le0$, the weak local
Poincar\'e inequality with means \cite[Theorem~1.2]{Rajala12}
applies; by the standard theory of Poincar\'e inequalities on
doubling spaces \cite[Theorem~5.1]{HajlaszKoskela00}, this
inequality improves to the following local
Sobolev--Poincar\'e estimate with means: 
\[
 \left(\fint_{B_s(z)}
       \Bigl|v-\fint_{B_s(z)}v\di\meas\Bigr|^t\di\meas
 \right)^{1/t}
 \le C s
       \left(\fint_{B_{2s}(z)}|Dv|^2\di\meas\right)^{1/2},
\]
where $t$ ranges as in \eqref{eq:irc-sobolev-moments} and, after
rescaling to unit size, $C=C(N,t,K_-R^2)$; for $N>2$ the
endpoint $t=2N/(N-2)$ is recorded in
\cite[Appendix~B,~(B.3)]{NobiliViolo24}, and the general form of
the estimate is stated in \cite[Theorem~2.14]{NobiliViolo22}.

Let $w\in W^{1,2}_0(B_r)$ and extend it by zero to $X$. By
geodesicity and $X\setminus B_{16R}(o)\ne\varnothing$, there exists
$z\in X$ with $\dist(x,z)=2r$. Since
$B_{r/4}(z)\subset B_{3r}(x)\setminus B_r(x)$, volume comparison
gives
\[
 \meas\bigl(B_{3r}(x)\setminus B_r(x)\bigr)
 \ge c\,\meas(B_{3r}(x)).
\]
Consequently, writing
$w_{B_{3r}}:=\fint_{B_{3r}(x)}w\di\meas$, we have
\[
 \norm{w}_{L^p(B_{3r}(x))}
 \le (1+c^{-1/p})
       \norm{w-w_{B_{3r}}}_{L^p(B_{3r}(x))},
 \qquad p\ge1.
\]
For $t\ge1$, choosing $p=t$ and applying the Sobolev--Poincar\'e
estimate on $B_{3r}(x)$, using that $|Dw|=0$ almost everywhere
outside $B_r(x)$ together with relative volume comparison, gives
\eqref{eq:irc-sobolev-moments} in this range, and taking $t=2$ gives
\eqref{eq:irc-poincare}. For $0<t<1$, Jensen's inequality for the
probability measure $\meas(B_r)^{-1}\meas|_{B_r}$ gives
\[
 \left(\fint_{B_r}|w|^t\di\meas\right)^{1/t}
 \le\fint_{B_r}|w|\di\meas
 \le\left(\fint_{B_r}|w|^2\di\meas\right)^{1/2},
\]
and \eqref{eq:irc-poincare} bounds the last term by
$Cr\bigl(\fint_{B_r}|Dw|^2\di\meas\bigr)^{1/2}$; hence
\eqref{eq:irc-sobolev-moments} holds in this range as well. Finally,
$t=q'$ gives \eqref{eq:irc-sobolev}.
\end{proof}
For localization, we use Lipschitz cutoff functions
$\chi\in\Lip_c(B_r(x))$ satisfying $0\le\chi\le1$,
$\chi=1$ on $B_{r/2}(x)$, and $|D\chi|\le C/r$.
Such cutoffs are available on every ball.

Let $U\subset X$ be open and let
$v\in W^{1,2}_{\loc}(U)$ be a weak solution of
$\Delta v=F$, with $F\in L^\infty_{\loc}(U)$.
Testing the equation with $\chi^2v$ and applying Young's
inequality gives the Caccioppoli estimate
\begin{equation}
 \int_U\chi^2|Dv|^2\di\meas
 \le4\int_Uv^2|D\chi|^2\di\meas
       +2\int_U\chi^2|Fv|\di\meas,
 \qquad \chi\in\Lip_c(U).
 \label{eq:irc-caccioppoli}
\end{equation}
For unbounded $v$, the test is justified by truncation.

\subsection{Sobolev maps and harmonic replacements}
\label{subsec:energyharmonics}

The quantities $|\mathrm du|$ and $e_2[u]$ correspond to two norms of the
classical differential. For a smooth map between Riemannian
manifolds with $n$-dimensional source, they are
\[
 |\mathrm du|(x)=\sup_{|\xi|=1}|\mathrm du_x(\xi)|,
 \qquad
 e_2[u](x)=\left(\frac1{n+2}
              \sum_{i=1}^n|\mathrm du_x(\xi_i)|^2\right)^{1/2},
\]
where $(\xi_i)_{i=1}^n$ is an orthonormal basis of the source tangent
space. We now recall the intrinsic definitions of these quantities for metric-valued Sobolev maps.

Let $U\subset X$ be bounded and open. The space $L^2(U,Y)$ consists
of essentially separably valued Borel maps, up to equality
\(\meas\)-almost everywhere, such that
$\dist_Y(u,y_*)\in L^2(U)$ for one, and hence every, $y_*\in Y$; it
is endowed with the metric
\[
 \dist_{L^2}(u,v)^2:=\int_U\dist_Y(u,v)^2\di\meas.
\]
Sobolev maps are defined by scalar post-composition, without using
target coordinates. This approach goes back to Ambrosio
\cite{Ambrosio90} for metric-valued $BV$ maps and to Reshetnyak
\cite{Reshetnyak97} for Sobolev maps, and we use the formulation on
open subsets of metric measure spaces given in
\cite[Definition~5.3]{GigliTulyenev21}: a map $u\in L^2(U,Y)$ belongs
to $W^{1,2}(U,Y)$ if
$\phi\circ u\in W^{1,2}(U)$ for every $1$-Lipschitz function
$\phi:Y\to\setR$ and
\[
 |\mathrm du|:=\operatorname*{ess\,sup}_{\substack{\phi:Y\to\setR\\
                                      \Lip(\phi)\le1}}
              |D(\phi\circ u)|\in L^2(U).
\]
It is
called the minimal $2$-weak upper gradient of $u$. Distance functions and
the scalar lattice calculus then give
\[
 |D\dist_Y(u,y_*)|\le |\mathrm du|,
 \qquad
 |D\dist_Y(u,v)|\le |\mathrm du|+|\mathrm dv|
\]
for Sobolev maps $u,v$. 
Gigli--Pasqualetto--Soultanis \cite{GigliPasqualettoSoultanis2020}
construct a differential $\mathrm du$
for every $u\in W^{1,2}(U,Y)$. It satisfies the chain rule for
scalar post-compositions, and its pointwise operator norm is the
minimal $2$-weak upper gradient introduced above
\cite[Definition~3.4 and Proposition~3.5]{GigliPasqualettoSoultanis2020}.
Thus $|\mathrm du|$ recovers the maximal stretching of the classical
differential.

The space $X$ is strongly rectifiable of essential dimension $n$:
for every $\varepsilon>0$, there is a countable Borel partition
$(U_i^\varepsilon)_i$ of $X$ up to a $\meas$-null set, with
$(1+\varepsilon)$-bi-Lipschitz charts
$\varphi_i^\varepsilon:U_i^\varepsilon\to\setR^n$ satisfying
\[
 c_i^\varepsilon\Leb^n\big|_{\varphi_i^\varepsilon(U_i^\varepsilon)}
 \le(\varphi_i^\varepsilon)_\#
       (\meas|_{U_i^\varepsilon})
 \le(c_i^\varepsilon+\varepsilon)
       \Leb^n\big|_{\varphi_i^\varepsilon(U_i^\varepsilon)}
\]
for some $c_i^\varepsilon>0$; see
\cite[Section~2.4]{GigliTulyenev21}.
We fix an aligned family $(\mathcal A^{\varepsilon_j})_{j\ge1}$ of
such atlases, with $\varepsilon_j\downarrow0$, in the sense of that
reference. This family is used for all maps below.
We use the notion of approximate metric differentiability from \cite[Definition~3.3]{GigliTulyenev21}, with the fixed aligned atlases restricted to $U$. For a Borel map $u:U\to Y$, we denote its approximate metric differential at $x$ by $s_{u,x}$ whenever it exists. This is a seminorm on $\mathbb R^n$.

Following \cite[Definition~5.4]{GigliTulyenev21}, we recall the
Korevaar--Schoen energy on open sets.

\begin{definition}[Korevaar--Schoen space]
\label{def:ks-space}
For $u\in L^2(U,Y)$ and $r>0$, set
\[
 \kse_{2,r}[u,U](x):=
 \begin{cases}
 \displaystyle\left(\fint_{B_r(x)}
       \frac{\dist_Y(u(x),u(y))^2}{r^2}\di\meas(y)\right)^{1/2},
       &B_r(x)\subset U,\\[4pt]
 0,&\text{otherwise}.
 \end{cases}
\]
The Korevaar--Schoen energy and Sobolev space are
\begin{equation}
 \begin{aligned}
 \mathrm E_2^U(u)&:=
 \sup_{\substack{\eta\in C_c(U)\\0\le\eta\le1}}
 \limsup_{r\downarrow0}\int_U\eta\,\kse_{2,r}[u,U]^2\di\meas,\\
 \ks^{1,2}(U,Y)&:=
 \{u\in L^2(U,Y):\mathrm E_2^U(u)<\infty\}.
 \end{aligned}
 \label{eq:ks-definition}
\end{equation}
\end{definition}
Recall that a complete geodesic space $(Y,\dist_Y)$ is $\CAT(0)$
if, along every constant-speed geodesic $\gamma:[0,1]\to Y$,
\[
 \dist_Y(p,\gamma_t)^2
 \le(1-t)\dist_Y(p,\gamma_0)^2+t\dist_Y(p,\gamma_1)^2
       -t(1-t)\dist_Y(\gamma_0,\gamma_1)^2
\]
for all $p\in Y$ and $t\in[0,1]$. Geodesics are unique; we write
$[p,q]_t$ for their constant-speed parametrization. We shall use
distance convexity and the four-point inequality
\[
 \begin{aligned}
 \dist_Y([p_0,p_1]_t,[q_0,q_1]_t)
 &\le(1-t)\dist_Y(p_0,q_0)+t\dist_Y(p_1,q_1),\\
 \dist_Y([p_0,p_1]_t,[q_0,q_1]_t)^2
 &\le(1-t)\dist_Y(p_0,q_0)^2+t\dist_Y(p_1,q_1)^2\\
 &\quad-t(1-t)
       \bigl(\dist_Y(p_0,p_1)-\dist_Y(q_0,q_1)\bigr)^2,
 \end{aligned}
\]
as well as the well-defined, nonexpansive projection onto any
nonempty closed convex subset of $Y$; see
\cite[Definition~6.1 and Eq.~(6.2)]{GigliTulyenev21} and
\cite[Propositions~II.2.2 and~II.2.4]{BridsonHaefliger99}.

For $u\in W^{1,2}(U,Y)$, set
\begin{equation}
 e_2[u]:=\frac{|\mathrm du|_{\mathrm{HS}}}{\sqrt{n+2}},
 \qquad
 E(u;U):=\frac{n+2}{2}\mathrm E_2^U(u).
 \label{eq:irc-normalization}
\end{equation}
The following theorem collects the basic properties of these
objects and the comparison of metric differentials used below.

\begin{theorem}
\label{thm:irc-sobolev-energy}
\label{lem:irc-differential-input}
Let $U\subset X$ be bounded and open.
\begin{enumerate}[label=\textup{(\roman*)},ref=\roman*]
\item\label{item:irc-energy-density}
The spaces $\ks^{1,2}(U,Y)$ and $W^{1,2}(U,Y)$ coincide.
For every $u$ in this space, $e_2[u]\in L^2(U)$ and, as $r\downarrow0$,
\begin{equation}
 \kse_{2,r}[u,U]\to e_2[u]
 \quad\meas\text{-a.e. and in }L^2_{\loc}(U),
 \qquad
 \mathrm E_2^U(u)=\int_U e_2[u]^2\di\meas.
 \label{eq:ks-density-limit}
\end{equation}
In particular, $E(u;U)=\frac{n+2}{2}\int_U e_2[u]^2\di\meas$.
The functional $\mathrm E_2^U:L^2(U,Y)\to[0,\infty]$ is lower
semicontinuous. If $V\subset U$ is open, then
$u|_V\in\ks^{1,2}(V,Y)$ and $e_2[u|_V]=e_2[u]$ almost everywhere on $V$.

\item\label{item:irc-metric-existence}
Every $u\in\ks^{1,2}(U,Y)$ is approximately metrically
differentiable almost everywhere in $U$, and
\[
 e_2[u](x)
 =\left(\fint_{B_1(0)\subset\setR^n}
          s_{u,x}(\xi)^2\di\Leb^n(\xi)\right)^{1/2}
 \qquad\text{for }\meas\text{-almost every }x\in U.
\]

\item\label{item:irc-metric-representation}
For every $u\in\ks^{1,2}(U,Y)$, there is a measurable field of
positive semidefinite symmetric $n\times n$ matrices $A_{u,x}$ such that, for
$\meas$-almost every $x\in U$ and every $\xi\in\setR^n$,
\begin{equation}
 s_{u,x}(\xi)^2=\langle A_{u,x}\xi,\xi\rangle.
 \label{eq:irc-seminorm}
\end{equation}
Let $\sigma_n$ be probability surface measure on $S^{n-1}$. Then, almost everywhere,
\begin{equation}
 \begin{aligned}
 e_2[u](x)
 &=\left(\frac{\operatorname{tr}A_{u,x}}{n+2}\right)^{1/2}
   =\left(\frac n{n+2}\int_{S^{n-1}}
             s_{u,x}(\xi)^2\di\sigma_n(\xi)\right)^{1/2},\\
 |\mathrm du|(x)&=\sqrt{\lambda_{\max}(A_{u,x})}
        =\max_{\xi\in S^{n-1}}s_{u,x}(\xi).
 \end{aligned}
 \label{eq:irc-spherical-trace}
\end{equation}
In particular, $|\mathrm du|^2\le(n+2)e_2[u]^2\le n |\mathrm du|^2$.
For scalar maps $f\in W^{1,2}(U)$,
$e_2[f]=(n+2)^{-1/2}|Df|$.

\item\label{item:irc-metric-comparison}
Let $u,v,w\in\ks^{1,2}(U,Y)$ and $\alpha,\beta\ge0$.
Suppose that
\begin{equation}
 \dist_Y(w(p),w(q))
 \le\alpha\dist_Y(u(p),u(q))
     +\beta\dist_Y(v(p),v(q))
 \label{eq:irc-metric-distance-comparison}
\end{equation}
holds for $(\meas\otimes\meas)$-almost every $(p,q)\in U\times U$.
Then, for $\meas$-almost every $x\in U$ and every $\xi\in\setR^n$,
\begin{equation}
 s_{w,x}(\xi)\le\alpha s_{u,x}(\xi)+\beta s_{v,x}(\xi).
 \label{eq:irc-metric-direction-comparison}
\end{equation}
\end{enumerate}
\end{theorem}
Statements \textup{(i)} and \textup{(ii)} follow from
\cite[Theorem~5.7 and Proposition~6.7]{GigliTulyenev21}.
For \textup{(iii)}, the identification of the approximate metric
differential with $\mathrm du$, together with the Hilbertianity
of its target module
\cite[Theorems~4.9, 4.12, and~6.5]{GigliTulyenev21},
gives the quadratic representation: in the orthonormal source
frame $(E_i)$ associated with the fixed atlases, take
$(A_{u,x})_{ij}
=\langle\mathrm du(E_i),\mathrm du(E_j)\rangle(x)$.
The identities follow from the ball and sphere averages
of quadratic forms and the pointwise operator norm.
For \textup{(iv)}, refine the charts into common Borel pieces
on which \(u,v,w\) are Lipschitz; see
\cite[Propositions~2.17 and~5.6]{GigliTulyenev21}.
At a.e. \(x\) in a piece, Fubini allows us to choose, in any
prescribed coordinate direction, sequences in that piece along
which \eqref{eq:irc-metric-distance-comparison} holds. Metric
differentiation then passes this inequality to the chart metric
differentials; their convergence along the fixed aligned atlases
\cite[Proposition~3.6]{GigliTulyenev21} then yields
\eqref{eq:irc-metric-direction-comparison}.

For $u\in\ks^{1,2}(U,Y)$, the prescribed-boundary-value class is
\begin{equation}
 \ks^{1,2}_u(U,Y)
 :=\{v\in\ks^{1,2}(U,Y):\dist_Y(u,v)\in W^{1,2}_0(U)\};
 \label{eq:dirichlet-class}
\end{equation}
see \cite[Definition~5.8]{GigliTulyenev21}.
A minimizer of $E(\cdot;U)$ in $\ks^{1,2}_u(U,Y)$ is called a
harmonic replacement of $u$ in $U$. A map $h\in\ks^{1,2}(U,Y)$ is
harmonic in $U$ if $h|_V$ minimizes $E(\cdot;V)$ over
$\ks^{1,2}_{h|_V}(V,Y)$ for every open $V\Subset U$.
Locality and Sobolev gluing show that every harmonic replacement
is harmonic in $U$.

\begin{lemma}
\label{lem:irc-convexity}
Let $B_r:=B_r(x)$ be admissible, and let
$u,v\in\ks^{1,2}(B_r,Y)$ satisfy
$w:=\dist_Y(u,v)\in W^{1,2}_0(B_r)$. For $0\le t\le1$, set
$u_t(z):=[u(z),v(z)]_t$. Then $u_t\in\ks^{1,2}_u(B_r,Y)$ and
\begin{equation}
 E(u_t;B_r)\le(1-t)E(u;B_r)+tE(v;B_r)
              -\frac{t(1-t)}2\int_{B_r}|Dw|^2\di\meas.
 \label{eq:irc-strong-convexity}
\end{equation}
\end{lemma}

\begin{proof}
The four-point inequality gives
\[
 \begin{aligned}
 \dist_Y(u_t(z),u_t(z'))^2
 &\le(1-t)\dist_Y(u(z),u(z'))^2+t\dist_Y(v(z),v(z'))^2\\
 &\quad-t(1-t)|w(z)-w(z')|^2.
 \end{aligned}
\]
Distance convexity gives $u_t\in L^2(B_r,Y)$, and the preceding
four-point inequality bounds its Korevaar--Schoen energy. Hence
$u_t\in\ks^{1,2}(B_r,Y)$; moreover,
$\dist_Y(u,u_t)=tw$ gives the boundary condition.
Averaging over balls and using \eqref{eq:ks-density-limit} yields
\[
 e_2[u_t]^2\le(1-t)e_2[u]^2+t e_2[v]^2
              -\frac{t(1-t)}{n+2}|Dw|^2
 \qquad\meas\text{-a.e. on }B_r,
\]
where $e_2[w]=(n+2)^{-1/2}|Dw|$. Integration and
\eqref{eq:irc-normalization} prove
\eqref{eq:irc-strong-convexity}; compare
\cite[Lemma~6.2]{GigliTulyenev21}.
\end{proof}
For a locally Lipschitz map $v:U\to Y$, we write
\[
 \lip v(x):=\limsup_{\substack{z\to x\\z\ne x}}
             \frac{\dist_Y(v(z),v(x))}{\dist(z,x)},
\]
with value zero at isolated points.
The following proposition records the existence, uniqueness and
regularity properties of harmonic replacements established in
\cite{GigliTulyenev21,GigliLipschitz}.

\begin{proposition}
\label{prop:irc-exact}
Let $B_r:=B_r(x)$ be admissible and let $u\in\ks^{1,2}(B_r,Y)$.
Then $\ks^{1,2}_u(B_r,Y)$ has a unique energy minimizer $h=h_{B_r}$.
It admits a locally Lipschitz representative in $B_r$, and
\begin{equation}
 |\mathrm dh|=\lip h\quad\meas\text{-a.e.},
 \qquad
 |\mathrm dh|^2\in W^{1,2}_{\loc}(B_r)\cap L^\infty_{\loc}(B_r).
 \label{eq:irc-harmonic-regularity}
\end{equation}
In the measure-valued sense,
\begin{equation}
 \frac12\Delta(|\mathrm dh|^2)\ge K |\mathrm dh|^2\meas
                       \ge-K_-|\mathrm dh|^2\meas.
 \label{eq:irc-bochner}
\end{equation}
Moreover,
\begin{equation}
 \operatorname*{ess\,sup}_{B_{r/2}(x)}|\mathrm dh|
 \le\frac{C}{r}\inf_{y\in Y}
       \left(\fint_{B_r(x)}\dist_Y(h,y)^2\di\meas\right)^{1/2},
 \label{eq:irc-exact-gradient}
\end{equation}
where $C=C(N,K_-R^2)$.
\end{proposition}

Existence and uniqueness follow from
\cite[Theorem~6.4]{GigliTulyenev21}; the normalization
\eqref{eq:irc-normalization} does not change the minimizer.
Local Lipschitz regularity, \eqref{eq:irc-bochner}, the
\(W^{1,2}_{\loc}\)-regularity in \eqref{eq:irc-harmonic-regularity},
and the gradient estimate \eqref{eq:irc-exact-gradient} follow
from \cite[Theorem~6.18]{GigliLipschitz}.
\cite[Proposition~6.11]{GigliLipschitz} identifies $\lip h$
with \(|\mathrm dh|\) almost everywhere.

The following lemma records a standard metric-valued consequence
of the local Sobolev-to-Lipschitz property. We include a short
proof to make the local-to-global and scalar-to-metric steps
explicit.

\begin{lemma}
\label{lem:irc-sobolev-lipschitz}
Let $u\in\ks^{1,2}(B_r(x),Y)$ and $M\ge0$.
If $|\mathrm du|\le M$ almost everywhere on $B_r(x)$, then $u$ has an
$M$-Lipschitz representative on $B_{r/2}(x)$.
\end{lemma}

\begin{proof}
Choose a separable subset $Y_0\subset Y$ containing $u(z)$ for
almost every $z$, and a countable dense set $(y_i)$ in $Y_0$.
The functions $f_i:=\dist_Y(u,y_i)$ belong to $W^{1,2}(B_r(x))$
and satisfy $|Df_i|\le |\mathrm du|\le M$.
By the local scalar Sobolev-to-Lipschitz property
\cite[Proposition~2.10]{GigliViolo2023}, each \(f_i\) admits a
continuous representative, still denoted by \(f_i\), which is
locally \(M\)-Lipschitz in \(B_r(x)\) and satisfies
\[
 |f_i(p)-f_i(q)|\le M\dist(p,q)
\]
whenever
\(\dist(p,q)\le\dist(p,\partial B_r(x))\).

We first extend this estimate to all pairs in \(B_{r/2}(x)\).
Let \(p,q\in B_{r/2}(x)\), and let
\(\gamma:[0,L]\to X\) be a minimizing geodesic from \(p\) to
\(q\). The two triangle-inequality bounds
\[
 \dist(x,\gamma(t))\le\dist(x,p)+t,
 \qquad
 \dist(x,\gamma(t))\le\dist(x,q)+L-t
\]
give
\[
 \dist(x,\gamma(t))
 \le\frac{\dist(x,p)+\dist(x,q)+L}{2}<r.
\]
Thus \(\gamma\) is compactly contained in \(B_r(x)\). Partition
it sufficiently finely that
\[
 \dist(\gamma(t_j),\gamma(t_{j+1}))
 \le\dist(\gamma(t_j),\partial B_r(x))
\]
on every subarc. Applying the local estimate on each subarc and
summing the resulting inequalities yields
\[
 |f_i(p)-f_i(q)|
 \le M\sum_j\dist(\gamma(t_j),\gamma(t_{j+1}))
 =M\dist(p,q).
\]
Hence every \(f_i\) is \(M\)-Lipschitz on \(B_{r/2}(x)\).

On a common full-measure subset $F\subset B_{r/2}(x)$,
we have $u(F)\subset Y_0$ and $f_i=\dist_Y(u,y_i)$ for every $i$,
so
\[
 \dist_Y(u(p),u(q))=\sup_i|f_i(p)-f_i(q)|
 \le M\dist(p,q)\qquad(p,q\in F).
\]
Full support of $\meas$ makes $F$ dense in $B_{r/2}(x)$.
Completeness of $Y$ therefore gives the required Lipschitz extension.
\end{proof}

\section{Proof of Theorem~\ref{thm:replacement-potential}}
\label{sec:proof-replacement}
\subsection{Harmonic replacement comparison}
\label{sec:irc-comparisons}

The Bochner inequality \eqref{eq:irc-bochner} applies to the squared
operator norm $|\mathrm d h_B|^2$. To use this estimate for $u$, we first
control the difference between $|\mathrm du|$ and $|\mathrm d h_B|$ in terms of the
replacement energy gap. The same argument also gives a comparison
between the energy densities.

\begin{proposition}
\label{prop:irc-operator-gap}
Let $B$ be admissible, let $u\in\ks^{1,2}(B,Y)$, and set $h=h_B$.
Then
\begin{align}
 \int_B\bigl|\,|\mathrm du|-|\mathrm dh|\,\bigr|^2\di\meas
    &\le2(n+2)G_u(B),\label{eq:irc-operator-gap}\\
 \int_B|e_2[u]-e_2[h]|^2\di\meas
    &\le\frac4{n+2}G_u(B).\label{eq:irc-trace-gap}
\end{align}
\end{proposition}

Its proof compares the stretches of $\mathrm du$ and
$\mathrm dh$ direction by direction: Busemann convexity along the
geodesic homotopy, together with the minimality of $h$, bounds the
differences of the directional stretches by the energy gap, and a
spherical moment estimate then passes to the maximal stretches.

\begin{lemma}
\label{lem:irc-directional}
Let $B$ be admissible, let $u\in\ks^{1,2}(B,Y)$, and set $h=h_B$.
For the seminorms in Theorem~\ref{thm:irc-sobolev-energy}, define
\[
 J(x):=\int_{S^{n-1}}
       |s_{u,x}(\xi)-s_{h,x}(\xi)|^2\di\sigma_n(\xi).
\]
Then
\[
 \int_B J\di\meas\le\frac4n G_u(B).
\]
\end{lemma}

\begin{proof}
The midpoint map $m=[u,h]_{1/2}$ belongs to $\ks^{1,2}_u(B,Y)$ by
Lemma~\ref{lem:irc-convexity}. Busemann convexity gives
\[
 \dist_Y(m(p),m(q))
 \le\frac12\dist_Y(u(p),u(q))
    +\frac12\dist_Y(h(p),h(q)).
\]
Passing to approximate metric differentials for $u$, $h$, and $m$
in common source coordinates, as in
Theorem~\ref{thm:irc-sobolev-energy}, yields
\[
 s_{m,x}(\xi)\le\frac12s_{u,x}(\xi)+\frac12s_{h,x}(\xi)
\]
for almost every $x$ and every $\xi\in\setR^n$. Squaring, we obtain
\[
 s_{m,x}(\xi)^2
 \le\frac12s_{u,x}(\xi)^2+\frac12s_{h,x}(\xi)^2
    -\frac14|s_{u,x}(\xi)-s_{h,x}(\xi)|^2.
\]
Integrating over $S^{n-1}$ and $B$, and using
\eqref{eq:irc-spherical-trace}, gives
\[
 E(m;B)\le\frac{E(u;B)+E(h;B)}2
               -\frac n8\int_B J\di\meas.
\]
Since $E(h;B)\le E(m;B)$ by minimality, rearranging gives the claim.
\end{proof}

We use the fourth-moment identity
\begin{equation}
 \int_{S^{n-1}}\langle Q\xi,\xi\rangle^2\di\sigma_n(\xi)
 =\frac{(\operatorname{tr}Q)^2+2\operatorname{tr}(Q^2)}{n(n+2)},
 \label{eq:irc-fourth-moment}
\end{equation}
valid for every symmetric matrix $Q$ on $\setR^n$;
see \cite{Folland2001}.
In particular, it applies to the possibly indefinite matrix
$Q=A-B$ used below.

\begin{lemma}
\label{lem:irc-spherical-operator}
Let $A,B$ be positive semidefinite symmetric operators on $\setR^n$,
and set
\[
 s_A(\xi)=\sqrt{\langle A\xi,\xi\rangle},
 \qquad
 s_B(\xi)=\sqrt{\langle B\xi,\xi\rangle}.
\]
Then
\[
 \left|\sqrt{\lambda_{\max}(A)}
       -\sqrt{\lambda_{\max}(B)}\right|^2
 \le\frac{n(n+2)}2
       \int_{S^{n-1}}|s_A-s_B|^2\di\sigma_n.
\]
\end{lemma}

\begin{proof}
Write $a=\sqrt{\lambda_{\max}(A)}$ and
$b=\sqrt{\lambda_{\max}(B)}$. By symmetry, assume $a\ge b$.
The case $a+b=0$ is immediate, so suppose $a+b>0$ and set $Q=A-B$.
For a unit eigenvector $\xi_0$ of $A$ corresponding to $a^2$,
\[
 \langle Q\xi_0,\xi_0\rangle
 =a^2-\langle B\xi_0,\xi_0\rangle\ge a^2-b^2.
\]
Thus $\|Q\|_{\mathrm{op}}\ge a^2-b^2$. The fourth-moment identity
\eqref{eq:irc-fourth-moment} gives
\[
 \begin{aligned}
 \int_{S^{n-1}}\langle Q\xi,\xi\rangle^2\di\sigma_n
 &=\frac{(\operatorname{tr}Q)^2+2\operatorname{tr}(Q^2)}{n(n+2)}\\
 &\ge\frac{2\|Q\|_{\mathrm{op}}^2}{n(n+2)}
 \ge\frac{2(a^2-b^2)^2}{n(n+2)}.
 \end{aligned}
\]
On the other hand, for $\xi\in S^{n-1}$,
\[
 |\langle Q\xi,\xi\rangle|
 =(s_A(\xi)+s_B(\xi))|s_A(\xi)-s_B(\xi)|
 \le(a+b)|s_A(\xi)-s_B(\xi)|.
\]
Combining these estimates and dividing by $(a+b)^2$ yields
\[
 \int_{S^{n-1}}|s_A-s_B|^2\di\sigma_n
 \ge\frac{2(a-b)^2}{n(n+2)},
\]
which is the desired inequality.
\end{proof}

\begin{proof}[Proof of Proposition~\ref{prop:irc-operator-gap}]
Let $J$ be as in Lemma~\ref{lem:irc-directional}. Applying
Lemma~\ref{lem:irc-spherical-operator} to $A_{u,x}$ and $A_{h,x}$,
using \eqref{eq:irc-spherical-trace}, and integrating over $B$, we find
\[
 \int_B\bigl|\,|\mathrm du|-|\mathrm dh|\,\bigr|^2\di\meas
 \le\frac{n(n+2)}2\int_B J\di\meas
 \le2(n+2)G_u(B).
\]
This is \eqref{eq:irc-operator-gap}.

For the energy densities, \eqref{eq:irc-spherical-trace} gives
$e_2[u](x)=\sqrt{\frac n{n+2}}\,\|s_{u,x}\|_{L^2(S^{n-1},\sigma_n)}$,
and the same identity holds for $h$. The triangle inequality in
$L^2(S^{n-1},\sigma_n)$ therefore yields
\[
 |e_2[u](x)-e_2[h](x)|^2\le\frac n{n+2}J(x).
\]
Integrating and applying Lemma~\ref{lem:irc-directional} gives
\eqref{eq:irc-trace-gap}.
\end{proof}

We next record the distance estimates that follow from energy
convexity. For real-valued Sobolev functions $f$ and $g$,
\[
 \bigl|\,|Df|-|Dg|\,\bigr|
 \le |D(f-g)|
 =|D\dist_{\setR}(f,g)|
 \qquad\meas\text{-a.e.}
\]
For metric-valued maps, the gradient of $\dist_Y(u,h)$ does not
directly control $\bigl|\,|\mathrm du|-|\mathrm dh|\,\bigr|$. The distance estimates in
Proposition~\ref{prop:irc-distance-gap} below therefore do not give
\eqref{eq:irc-operator-gap}. 

\begin{proposition}
\label{prop:irc-distance-gap}
Let $B=B_r(x)$ be admissible, let $u\in\ks^{1,2}(B,Y)$, and set
$h=h_B$. Then $w=\dist_Y(u,h)\in W^{1,2}_0(B)$ and
\begin{equation}
 \int_B|Dw|^2\di\meas\le2G_u(B),\qquad
 \left(\fint_B w^2\di\meas\right)^{1/2}
       \le Cr\,\delta_u(B).
 \label{eq:irc-distance-gap}
\end{equation}
If $q_*(N)<q<N$ and $q'=q/(q-1)$, then
\begin{equation}
 \left(\fint_B w^{q'}\di\meas\right)^{1/q'}
       \le C_{N,q,K_-R^2}r\,\delta_u(B).
 \label{eq:irc-distance-qprime}
\end{equation}
For the same range of $q$, suppose that $f\in L^q(B)$ is nonnegative
and
\begin{equation}
 E(u;B)\le E(v;B)+\int_B f\dist_Y(u,v)\di\meas
 \qquad\text{for every }v\in\ks^{1,2}_u(B,Y).
 \label{eq:irc-tension-hypothesis}
\end{equation}
Then
\begin{equation}
 \delta_u(B)\le C_{N,q,K_-R^2}r
                   \left(\fint_B f^q\di\meas\right)^{1/q}.
 \label{eq:irc-tension-majorant}
\end{equation}
\end{proposition}

\begin{proof}
The boundary condition gives $w\in W^{1,2}_0(B)$. For the geodesic
homotopy $u_t=[u,h]_t$, Lemma~\ref{lem:irc-convexity} and
the minimality of $h$ imply
\[
 0\le E(u_t;B)-E(h;B)
 \le(1-t)G_u(B)-\frac{t(1-t)}2\int_B|Dw|^2\di\meas,
 \qquad 0<t<1.
\]
Dividing by $1-t$ and letting $t\uparrow1$ gives
$\int_B|Dw|^2\di\meas\le2G_u(B)$. The remaining distance estimates
follow from \eqref{eq:irc-poincare} and \eqref{eq:irc-sobolev}.

Taking $v=h$ in \eqref{eq:irc-tension-hypothesis} and applying
H\"older's inequality and \eqref{eq:irc-distance-qprime}, we obtain
\[
 \begin{aligned}
 G_u(B)
\le\int_B fw\di\meas
 &\le\meas(B)
       \left(\fint_B f^q\di\meas\right)^{1/q}
       \left(\fint_B w^{q'}\di\meas\right)^{1/q'}\\
 &\le Cr\,\meas(B)^{1/2}G_u(B)^{1/2}
       \left(\fint_B f^q\di\meas\right)^{1/q}.
 \end{aligned}
\]
If $G_u(B)>0$, division by
$\meas(B)^{1/2}G_u(B)^{1/2}$ gives
\eqref{eq:irc-tension-majorant}; the case $G_u(B)=0$ is immediate.
\end{proof}

\subsection{The iteration estimate}
\label{subsec:gbl-output}

We now use Proposition~\ref{prop:irc-operator-gap} to compare
local averages of $|\mathrm du|$ at successive scales. The averages are
taken against the probability measures of
Theorem~\ref{thm:gbl-averages}, supported on thick Green bands in
$X\times\setR^3$, the counterparts of the averages over Green
level sets in the classical mean-value formulas
\cite{Watson91,Ni07}. We first
record a weighted version of the replacement comparison.

\begin{corollary}
\label{cor:irc-weighted}
Let $B$ be admissible, let $u\in\ks^{1,2}(B,Y)$, and set $h=h_B$.
Let $(Z,\sigma)$ be a probability space and let $\pi:Z\to B$ be
measurable, with
\[
 \pi_\#\sigma\le\frac{\Lambda}{\meas(B)}\,\meas|_B
\]
for some $\Lambda>0$. If $b\in L^\infty(Z,\sigma)$ is nonnegative
and $M_b:=\|b\|_{L^\infty(Z,\sigma)}$, then
\begin{equation}
 \left|
  \left(\int_Z b\,(|\mathrm du|\circ\pi)^2\di\sigma\right)^{1/2}
  -\left(\int_Z b\,(|\mathrm dh|\circ\pi)^2\di\sigma\right)^{1/2}
 \right|
 \le\sqrt{2(n+2)M_b\Lambda}\,\delta_u(B).
 \label{eq:irc-weighted-comparison}
\end{equation}
\end{corollary}

\begin{proof}
By the triangle inequality, the left-hand side of
\eqref{eq:irc-weighted-comparison} is bounded by
\[
 \begin{aligned}
 \left(\int_Z b\,\Bigl|\bigl(|\mathrm du|-|\mathrm dh|\bigr)\circ\pi\Bigr|^2\di\sigma\right)^{1/2}
 &\le\left(\frac{M_b\Lambda}{\meas(B)}
                 \int_B\bigl|\,|\mathrm du|-|\mathrm dh|\,\bigr|^2\di\meas\right)^{1/2}\\
 &\le\sqrt{2(n+2)M_b\Lambda}\,\delta_u(B),
 \end{aligned}
\]
where the last inequality follows from
Proposition~\ref{prop:irc-operator-gap}.
\end{proof}

We next state the properties of the averaging measures. Their
construction depends only on the source geometry and scalar
analysis, and is given in Section~\ref{sec:gbl-averages}.
Fix $x\in B_R(o)$ and set
\[
 Z=X\times\setR^3,\qquad
 \nu=\meas\otimes\Leb^3,\qquad
 P=(x,0),\qquad
 \pi_X(y,z)=y.
\]
We equip $Z$ with the product metric $\dist_Z$. The constants
below depend only on $N$ and $K_-R^2$.

\begin{theorem}
\label{thm:gbl-averages}
There exists an integer $m\ge3$, depending only on $N$ and
$K_-R^2$, such that, setting
\[
 s_j=2^{-mj}R/64,\qquad j\ge0,
\]
there is a sequence of probability measures $(\sigma_j)_{j\ge0}$
on $Z$, with bounded, compactly supported densities with respect
to $\nu$, satisfying the following properties.

\smallskip\noindent
\textup{(i) Localization and projected density.}
Each $\sigma_j$ is concentrated on
\[
 \{s_{j+1}<\dist_Z(P,\cdot)<s_j\},
\]
and its support is contained in the corresponding closed annulus.
Moreover,
\begin{equation}
 (\pi_X)_\#\sigma_j
 \le\frac{C\,\mathbf1_{B_{s_j}(x)}}{\meas(B_{s_j}(x))}\,\meas,
 \label{eq:gbl-projection}
\end{equation}
with $C$ independent of $j$ and $x\in B_R(o)$.

\smallskip\noindent
\textup{(ii) Subharmonic monotonicity.}
If $U$ is an open neighborhood of $\overline{B^Z_{s_j}(P)}$ and
$w\in W^{1,2}_{\loc}(U)$ is weakly subharmonic in $U$, then
\begin{equation}
 \int_Z w\di\sigma_{j+1}\le\int_Z w\di\sigma_j.
 \label{eq:gbl-monotonicity}
\end{equation}

\smallskip\noindent
\textup{(iii) Recovery at Lebesgue points.}
If $F\in L^1_{\loc}(X)$ and $x$ is a Lebesgue point of $F$, then
\begin{equation}
 \int_Z F(y)\di\sigma_j(y,z)\to F(x).
 \label{eq:gbl-differentiation}
\end{equation}
\end{theorem}

The projected-density bound allows us to apply the weighted
comparison on two consecutive scales. Subharmonic monotonicity
then gives the desired recursion.

\begin{corollary}
\label{cor:gbl-recursion}
Let $u\in\ks^{1,2}(\Omega,Y)$ and let $(\sigma_j)$ be as in
Theorem~\ref{thm:gbl-averages}. Set
\[
 \widehat B_j=B_{4s_j}(x),\qquad
 h_j=h_{\widehat B_j},\qquad
 b(z)=\cosh\bigl(\sqrt{2K_-}\,z_1\bigr),
\]
and define
\[
 A_j:=\left(\int_Z b(z)|\mathrm du|^2(y)\di\sigma_j(y,z)\right)^{1/2}.
\]
Then $A_j<\infty$ for every $j\ge0$ and
\begin{equation}
 A_{j+1}\le A_j+C\,\delta_u(\widehat B_j).
 \label{eq:gbl-recursion}
\end{equation}
If $x$ is a Lebesgue point of $|\mathrm du|^2$, then
\begin{equation}
 A_j\to |\mathrm du|(x).
 \label{eq:gbl-recursion-limit}
\end{equation}
\end{corollary}

\begin{proof}
Fix $j\ge0$. We compare both $A_j$ and $A_{j+1}$ with averages
of the same harmonic replacement $h_j$. For $k=j,j+1$, set
\[
 H_{j,k}:=\left(\int_Z b(z)|\mathrm d h_j|^2(y)
                         \di\sigma_k(y,z)\right)^{1/2}.
\]
On the supports of these measures, $1\le b\le C$. Moreover,
$B_{s_k}(x)\subset\widehat B_j$, and the ratios $4s_j/s_k$ are
$4$ and $2^{m+2}$, respectively. Thus relative volume comparison and
\eqref{eq:gbl-projection} give
\[
 (\pi_X)_\#\sigma_k
 \le\frac{C\,\mathbf1_{\widehat B_j}}{\meas(\widehat B_j)}\,\meas,
 \qquad k=j,j+1.
\]
In particular, $A_k$ and $H_{j,k}$ are finite. Applying
Corollary~\ref{cor:irc-weighted} on $\widehat B_j$, we obtain
\begin{equation}
 |A_k-H_{j,k}|\le C\,\delta_u(\widehat B_j),
 \qquad k=j,j+1.
 \label{eq:gbl-two-scale-comparison}
\end{equation}

We claim that $W_j(y,z):=b(z)|\mathrm d h_j|^2(y)$ is weakly
subharmonic in $U_j:=\widehat B_j\times\setR^3$. Indeed,
Proposition~\ref{prop:irc-exact} gives
\[
 |\mathrm d h_j|^2\in W^{1,2}_{\loc}(\widehat B_j)
              \cap L^\infty_{\loc}(\widehat B_j),\qquad
 \Delta_X(|\mathrm d h_j|^2)\ge-2K_-|\mathrm d h_j|^2\,\meas,
\]
while $\Delta_{\setR^3}b=2K_-b$. Since the factors depend on
different variables, the product rule gives, for every
nonnegative $\varphi\in\Lip_c(U_j)$,
\[
 -\int_{U_j}\langle DW_j,D\varphi\rangle\di\nu
 \ge\int_{U_j}(-2K_-b+\Delta_{\setR^3}b)
                   |\mathrm d h_j|^2\varphi\di\nu=0.
\]
As $U_j$ contains $\overline{B^Z_{s_j}(P)}$,
Theorem~\ref{thm:gbl-averages}\textup{(ii)} applies to $W_j$ and yields
\[
 H_{j,j+1}\le H_{j,j}.
\]
Combining this with \eqref{eq:gbl-two-scale-comparison}, we find
\[
 \begin{aligned}
 A_{j+1}
 &\le H_{j,j+1}+C\,\delta_u(\widehat B_j)\\
 &\le H_{j,j}+C\,\delta_u(\widehat B_j)\\
 &\le A_j+2C\,\delta_u(\widehat B_j).
 \end{aligned}
\]
This gives \eqref{eq:gbl-recursion} after enlarging $C$.

Finally, suppose that $x$ is a Lebesgue point of $|\mathrm du|^2$.
Apply Theorem~\ref{thm:gbl-averages}\textup{(iii)} to the zero
extension of $\left.|\mathrm du|^2\right|_{B_{2R}(o)}$. Since all projected supports
lie in $B_{2R}(o)$, this gives
\[
 \int_Z |\mathrm du|^2(y)\di\sigma_j(y,z)\to |\mathrm du|^2(x).
\]
The localization in \textup{(i)} also implies
$\|b-1\|_{L^\infty(\sigma_j)}\to0$. Hence
\[
 \begin{aligned}
 0\le A_j^2-\int_Z |\mathrm du|^2(y)\di\sigma_j(y,z)
 &\le\|b-1\|_{L^\infty(\sigma_j)}
        \int_Z |\mathrm du|^2(y)\di\sigma_j(y,z) \to0.
 \end{aligned}
\]
Taking square roots yields \eqref{eq:gbl-recursion-limit}.
\end{proof}

Since $4s_j=2^{-4-mj}R$, the balls $\widehat B_j$ occur among
the dyadic balls defining the replacement potential. In particular,
\[
 \sum_{j=0}^{\infty}\delta_u(\widehat B_j)\le\mathcal P_Ru(x).
\]
We shall combine this bound with \eqref{eq:gbl-recursion} and an
estimate for $A_0$ in the next subsection.

\subsection{Proof of the replacement-potential estimate}
\label{sec:proof-main-estimate}

We complete the proof of Theorem~\ref{thm:replacement-potential}.
The first step is to estimate the initial average $A_0$ in terms
of the energy and, when finite, the oscillation of $u$.

\begin{lemma}
\label{lem:main-initialization}
With the notation of Corollary~\ref{cor:gbl-recursion}, there is
a constant $C=C(N,K_-R^2)$ such that
\begin{equation}
 A_0\le C\,\mathcal E_{2R}(u;o).
 \label{eq:main-energy-initialization}
\end{equation}
If $O_x:=\operatorname*{ess\,osc}_{B_R(x)}u<\infty$, then
\begin{equation}
 \mathcal E_{R/2}(u;x)
 \le C\left[\frac{O_x}{R}+\delta_u(B_R(x))\right].
 \label{eq:main-local-energy-initialization}
\end{equation}
and
\begin{equation}
 A_0\le C\left[\frac{O_x}{R}+\delta_u(B_R(x))\right],
 \label{eq:main-osc-initialization}
\end{equation}
\end{lemma}

\begin{proof}
Recall that $s_0=R/64$, and that $b\le C$ on
$\operatorname{supp}\sigma_0$ because $|z_1|\le s_0$ there. By
\eqref{eq:gbl-projection} and
$|\mathrm du|^2\le(n+2)e_2[u]^2$,
\begin{equation}
 A_0^2\le C\fint_{B_{s_0}(x)}|\mathrm du|^2\di\meas
        \le C(n+2)\fint_{B_{s_0}(x)}e_2[u]^2\di\meas.
 \label{eq:main-A0-local-energy}
\end{equation}
The inclusions
\[
 B_{s_0}(x)\subset B_{2R}(o)\subset B_{3R}(x),
 \qquad s_0=R/64,
\]
together with Bishop--Gromov comparison, give
$\meas(B_{2R}(o))\le C\,\meas(B_{s_0}(x))$; combined with
\eqref{eq:main-A0-local-energy}, this yields
\eqref{eq:main-energy-initialization} after taking square roots.
The same argument, with $B_{R/2}(x)$ in place of $B_{2R}(o)$, yields
\[
 A_0\le C\,\mathcal E_{R/2}(u;x).
\]
It therefore suffices to prove
\eqref{eq:main-local-energy-initialization} under the assumption
$O_x<\infty$.

Let $h=h_{B_R(x)}$ and choose $y_0$ in the closed essential range
of $u|_{B_R(x)}$. Then $\dist_Y(u,y_0)\le O_x$ almost everywhere
on $B_R(x)$. Let
\(\mathcal C:=\overline B_{O_x}(y_0)\), which is a nonempty closed
convex subset of \(Y\), and let \(P_{\mathcal C}\) be the
nonexpansive projection onto \(\mathcal C\). The map
\(\widetilde h:=P_{\mathcal C}h\) belongs to
\(\ks^{1,2}(B_R(x),Y)\), satisfies
\(E(\widetilde h;B_R(x))\le E(h;B_R(x))\), and
\[
 \dist_Y(\widetilde h,u)\le\dist_Y(h,u)
 \qquad\meas\text{-a.e. on }B_R(x).
\]
Since \(h\in\ks^{1,2}_u(B_R(x),Y)\), we have
\(\dist_Y(h,u)\in W^{1,2}_0(B_R(x))\).
By \eqref{eq:irc-wzero-domination},
\(\dist_Y(\widetilde h,u)\in W^{1,2}_0(B_R(x))\) as well.
Hence \(\widetilde h\in\ks^{1,2}_u(B_R(x),Y)\), and uniqueness
of the harmonic replacement gives \(\widetilde h=h\) almost
everywhere. In particular,
\[
 \dist_Y(h,y_0)\le O_x
 \qquad\meas\text{-a.e. on }B_R(x).
\]
The gradient estimate \eqref{eq:irc-exact-gradient} consequently
implies
\begin{equation}
 \operatorname*{ess\,sup}_{B_{R/2}(x)}|\mathrm dh|
 \le C\frac{O_x}{R}.
 \label{eq:main-h-gradient}
\end{equation}
Since $(n+2)e_2[h]^2\le n |\mathrm dh|^2$, it follows that
\begin{equation}
 \mathcal E_{R/2}(h;x)\le C\frac{O_x}{R}.
 \label{eq:main-h-energy}
\end{equation}
The triangle inequality in $L^2(B_{R/2}(x))$ and
\eqref{eq:irc-trace-gap} now give
\[
 \begin{aligned}
 \mathcal E_{R/2}(u;x)
 &\le\mathcal E_{R/2}(h;x)
   +\sqrt{n+2}\left(\fint_{B_{R/2}(x)}
           |e_2[u]-e_2[h]|^2\di\meas\right)^{1/2}\\
 &\le C\frac{O_x}{R}
   +2\left(\frac{G_u(B_R(x))}{\meas(B_{R/2}(x))}\right)^{1/2}\\
 &\le C\left[\frac{O_x}{R}+\delta_u(B_R(x))\right],
 \end{aligned}
\]
where the last inequality follows from relative volume comparison
applied to $B_{R/2}(x)\subset B_R(x)$.
This proves \eqref{eq:main-local-energy-initialization}, and the
preceding bound for $A_0$ gives \eqref{eq:main-osc-initialization}.
\end{proof}

\begin{proof}[Proof of Theorem~\ref{thm:replacement-potential}]
Fix a Lebesgue point $x\in B_R(o)$ of $|\mathrm du|^2$. Such points form
a set of full measure, since $e_2[u]\in L^2(\Omega)$ and
$|\mathrm du|^2\le(n+2)e_2[u]^2$.
The pointwise bounds are immediate if $\mathcal P_Ru(x)=\infty$,
so assume that $\mathcal P_Ru(x)<\infty$.

Iterating \eqref{eq:gbl-recursion}, we obtain
\begin{equation}
 A_\ell\le A_0+C\sum_{j=0}^{\ell-1}\delta_u(\widehat B_j),
 \qquad \ell\ge1.
 \label{eq:main-finite-iteration}
\end{equation}
As observed in the preceding subsection, $4s_j=2^{-4-mj}R$,
and hence
\begin{equation}
 \sum_{j=0}^{\infty}\delta_u(\widehat B_j)
 =\sum_{j=0}^{\infty}
       \delta_u\bigl(B_{2^{-(4+mj)}R}(x)\bigr)
 \le\mathcal P_Ru(x).
 \label{eq:main-subsequence-potential}
\end{equation}
Moreover, Corollary~\ref{cor:gbl-recursion} gives
\begin{equation}
 A_\ell\to |\mathrm du|(x)
 \qquad\text{as }\ell\to\infty.
 \label{eq:main-A-limit}
\end{equation}
Letting $\ell\to\infty$ in \eqref{eq:main-finite-iteration}, we
deduce that
\[
 |\mathrm du|(x)\le A_0+C\mathcal P_Ru(x).
\]
Together with \eqref{eq:main-energy-initialization}, this yields
\begin{equation}
 |\mathrm du|(x)\le C\bigl[\mathcal E_{2R}(u;o)+\mathcal P_Ru(x)\bigr].
 \label{eq:main-pointwise-energy}
\end{equation}

For the oscillation bound, assume that
$\operatorname*{ess\,osc}_{B_{2R}(o)}u<\infty$. Since
$B_R(x)\subset B_{2R}(o)$,
\begin{equation}
 O_x\le\operatorname*{ess\,osc}_{B_{2R}(o)}u.
 \label{eq:main-local-global-osc}
\end{equation}
The initial term in the replacement potential also gives
\begin{equation}
 \delta_u(B_R(x))\le\mathcal P_Ru(x).
 \label{eq:main-initial-defect-potential}
\end{equation}
Using \eqref{eq:main-osc-initialization} in the same argument,
we obtain
\begin{equation}
 |\mathrm du|(x)\le C\left[
   \frac{\operatorname*{ess\,osc}_{B_{2R}(o)}u}{R}
   +\mathcal P_Ru(x)\right].
 \label{eq:main-pointwise-oscillation}
\end{equation}
Both pointwise estimates thus hold at every Lebesgue point of
$|\mathrm du|^2$ in $B_R(o)$.

Suppose now that
$M_R:=\sup_{x\in B_R(o)}\mathcal P_Ru(x)<\infty$.
By \eqref{eq:main-pointwise-energy},
\[
 |\mathrm du|\le C\bigl[\mathcal E_{2R}(u;o)+M_R\bigr]
 \qquad\meas\text{-a.e. on }B_R(o).
\]
Lemma~\ref{lem:irc-sobolev-lipschitz} provides a representative
$\widetilde u$ on $B_{R/2}(o)$ satisfying
\[
 \Lip\bigl(\widetilde u|_{B_{R/2}(o)}\bigr)
 \le C\bigl[\mathcal E_{2R}(u;o)+M_R\bigr].
\]
When the oscillation is finite, applying the same lemma to
\eqref{eq:main-pointwise-oscillation} gives the corresponding
Lipschitz bound. This completes the proof.
\end{proof}

The preceding proof also gives the following local form of the
oscillation estimate.

\begin{corollary}
\label{cor:main-local-form}
Let $u\in\ks^{1,2}(\Omega,Y)$, and let $z\in X$ and $r>0$ satisfy
\begin{equation}
 B_{4r}(z)\Subset\Omega,\qquad
 X\setminus B_{8r}(z)\neq\varnothing.
 \label{eq:main-local-setting}
\end{equation}
Assume that $z$ is a Lebesgue point of $|\mathrm du|^2$ and that
$\operatorname*{ess\,osc}_{B_r(z)}u<\infty$. Set
\[
 \mathcal P_r u(z)
 :=\sum_{k=0}^{\infty}\delta_u\bigl(B_{2^{-k}r}(z)\bigr).
\]
Then
\begin{equation}
 |\mathrm du|(z)\le C\left[
   \frac{\operatorname*{ess\,osc}_{B_r(z)}u}{r}
   +\mathcal P_r u(z)\right],
 \label{eq:main-local-pointwise}
\end{equation}
where $C=C(N,K_-r^2)$.
\end{corollary}

\begin{proof}
Apply the proof of Theorem~\ref{thm:replacement-potential} with $o$
replaced by $z$ and $R$ replaced by $R':=r/2$. The geometric
hypotheses follow from \eqref{eq:main-local-setting}, since
$8R'=4r$ and $16R'=8r$. Since the pointwise estimate holds at every
Lebesgue point of $|\mathrm du|^2$, its oscillation form gives
\[
 |\mathrm du|(z)\le C\left[
   \frac{2\,\operatorname*{ess\,osc}_{B_r(z)}u}{r}
   +\mathcal P_{r/2}u(z)\right].
\]
Finally,
\[
 \mathcal P_{r/2}u(z)
 =\sum_{k=1}^{\infty}\delta_u\bigl(B_{2^{-k}r}(z)\bigr)
 \le\mathcal P_r u(z),
\]
and \eqref{eq:main-local-pointwise} follows after absorbing the
factor $2$ into $C$.
\end{proof}

\section{Proof of Theorem~\ref{thm:gbl-averages}}
\label{sec:gbl-averages}

This section proves Theorem~\ref{thm:gbl-averages} using the scalar
Sobolev calculus, Gaussian heat-kernel bounds, and interior gradient
estimates recalled in Section~\ref{subsec:ms-background}.
The construction is the band counterpart of the classical mean-value formulas
over Green level sets \cite{Watson91,Ni07}. 
The factor $\setR^3$ makes the large-time heat-kernel tail integrable.
The use of an auxiliary Euclidean factor in Green-function
analysis on $\RCD$ spaces also appears in
\cite[Remark~2.1 and Section~2.3]{BrueSemola20}.

We keep the notation of Theorem~\ref{thm:gbl-averages}: $x\in B_R(o)$
is fixed and
\[
 Z:=X\times\setR^3,
 \qquad \nu:=\meas\otimes\Leb^3,
 \qquad P:=(x,0),
 \qquad \pi_X(y,z):=y.
\]
Recall from Section~\ref{subsec:ms-background} that
$(Z,\dist_Z,\nu)$ is $\RCD(-K_-,N+3)$, with product heat kernel
\eqref{eq:irc-product-kernel}. Unless an additional parameter is
indicated, constants depend only on $N$ and $K_-R^2$.

\subsection{The product Green kernel}
\label{subsec:gbl-kernel}

We first record the volume growth of product balls and the monotonicity
of the associated gauge
\[
 V_x(r):=\nu(B_r^Z(P)),
 \qquad
 \Phi_x(r):=\frac{r^2}{V_x(r)},
\]
which will measure the size of the Green kernel.

\begin{lemma}
\label{lem:gbl-volume}
There are constants $c,C>0$, depending only on $N$ and $K_-R^2$, such
that
\begin{equation}
 c\,r^3\meas(B_r(x))\le V_x(r)\le C\,r^3\meas(B_r(x)),
 \qquad 0<r\le5R.
 \label{eq:gbl-volume-local}
\end{equation}
Moreover, for all $0<s\le r$,
\begin{equation}
 V_x(r)\ge(r/s)^3V_x(s),
 \qquad
 \Phi_x(r)\le(s/r)\Phi_x(s),
 \label{eq:gbl-gauge}
\end{equation}
and $\Phi_x(r)\to\infty$ as $r\downarrow0$.
\end{lemma}

\begin{proof}
Writing $\omega_3:=\Leb^3(B_1(0))$ and integrating over the Euclidean
fibers, we have
\[
 V_x(r)=\omega_3\int_{B_r(x)}
 \bigl(r^2-\dist(x,y)^2\bigr)^{3/2}\di\meas(y).
\]
The upper bound in \eqref{eq:gbl-volume-local} holds for all $r>0$.
Restricting the integral to $B_{r/2}(x)$ and applying Bishop--Gromov
doubling for $r\le5R$ gives the lower bound. For $\dist(x,y)<s\le r$,
\[
 \bigl(r^2-\dist(x,y)^2\bigr)^{3/2}
 \ge(r/s)^3\bigl(s^2-\dist(x,y)^2\bigr)^{3/2};
\]
integrating over $B_s(x)$ gives the first inequality in
\eqref{eq:gbl-gauge}, and the second follows from the definition of
$\Phi_x$. Finally, taking $r=R$ in \eqref{eq:gbl-gauge} gives
$\Phi_x(s)\ge(R/s)\Phi_x(R)\to\infty$ as $s\downarrow0$.
\end{proof}

We next introduce the Green kernel of the product space and collect
its properties. By \eqref{eq:irc-product-kernel},
\[
 p_t^Z(P,(y,z))=p_t^X(x,y)(4\pi t)^{-3/2}e^{-|z|^2/(4t)}.
\]
Define
\begin{equation}
 G_P(\zeta):=\int_0^\infty p_t^Z(P,\zeta)\di t
 \quad(\zeta\ne P),
 \qquad G_P(P):=+\infty,
 \label{eq:gbl-green}
\end{equation}
and, for $\eps>0$, set
$G_\eps:=\int_\eps^\infty p_t^Z(P,\cdot)\di t$.

\begin{proposition}
\label{prop:gbl-green}
The function $G_P$ is finite, locally Lipschitz, and weakly harmonic
on $Z\setminus\{P\}$. It tends to $+\infty$ at $P$ and to zero
as $\dist_Z(P,\cdot)\to\infty$; in particular, $\{G_P\ge a\}$ is
compact for every $a>0$. Writing $r=\dist_Z(P,\zeta)$, we have
\begin{equation}
 \begin{aligned}
 c\Phi_x(r)&\le G_P(\zeta)\le C\Phi_x(r)
 &&(0<r\le3R),\\
 |DG_P|(\zeta)&\le C\frac{\Phi_x(r)}r
 &&(0<r\le2R),
 \end{aligned}
 \label{eq:gbl-green-bounds}
\end{equation}
where the gradient bound holds $\nu$-almost everywhere. Moreover,
$G_\eps\in W^{1,2}_{\loc}(Z)$ satisfies
\begin{equation}
 -\Delta_ZG_\eps=p_\eps^Z(P,\cdot)
 \quad\text{weakly on }Z,
 \label{eq:gbl-regularized}
\end{equation}
and, as $\eps\downarrow0$, $G_\eps\to G_P$ locally uniformly on
$Z\setminus\{P\}$ and strongly in $W^{1,2}_{\loc}(Z\setminus\{P\})$.
\end{proposition}

\begin{proof}
\medskip\noindent
\textbf{Pointwise bounds.}
Combining the Gaussian bounds \eqref{eq:irc-gaussian}, with the product formula
\eqref{eq:irc-product-kernel} and the volume comparison \eqref{eq:gbl-volume-local}, gives the two-sided
estimate
\begin{equation}
 \begin{aligned}
 \frac{c}{V_x(\sqrt t)}
 \exp\!\left(-\frac{\dist_Z(P,\zeta)^2}{ct}\right)
 \le p_t^Z(P,\zeta)
 \le\frac{C}{V_x(\sqrt t)}
 \exp\!\left(-\frac{\dist_Z(P,\zeta)^2}{Ct}\right),
 \qquad 0<t\le25R^2,
 \end{aligned}
 \label{eq:gbl-gaussian}
\end{equation}
where \(c,C>0\) depend only on \(N,K_-R^2\). 
For \(0<r=\dist_Z(P,\zeta)\le3R\), we have
\([r^2,2r^2]\subset(0,25R^2]\).
Integrating the lower bound in \eqref{eq:gbl-gaussian} over this interval and using
local doubling therefore gives
\[
 G_P(\zeta)\ge cr^2/V_x(r)=c\Phi_x(r).
\]

For $t\ge R^2$, the semigroup property, $L^\infty$-contraction,
and the Gaussian upper bound on $X$ at time $R^2$ give
\[
 \|p_t^X(x,\cdot)\|_\infty=\|P_{t-R^2}p_{R^2}^X(x,\cdot)\|_\infty\le M:=\|p_{R^2}^X(x,\cdot)\|_\infty\le\frac{C}{\meas(B_R(x))},
\]
and the product formula yields the tail estimate
\begin{equation}
 \sup_{\zeta\in Z}\int_T^\infty p_t^Z(P,\zeta)\di t
 \le CMT^{-1/2},\qquad T\ge R^2.
 \label{eq:gbl-tail}
\end{equation}
In particular, the tail from $R^2$ is bounded by
$C/[R\meas(B_R(x))]\le C\Phi_x(R)$, where the last inequality uses
\eqref{eq:gbl-volume-local}.

If $r\le R$, Bishop--Gromov comparison on \(Z\), whose dimension
bound is \(N+3\), gives
\[
 V_x(\sqrt t)\ge c\left(\frac{\sqrt t}{r}\right)^{N+3}V_x(r)
 \qquad(0<t\le r^2),
\]
and hence
\[
 \begin{aligned}
 \int_0^{r^2}p_t^Z(P,\zeta)\di t
 &\le\frac{C}{V_x(r)}\int_0^{r^2}
       \left(\frac r{\sqrt t}\right)^{N+3}e^{-r^2/(Ct)}\di t
 \le C\Phi_x(r),\\
 \int_{r^2}^{R^2}p_t^Z(P,\zeta)\di t
 &\le\frac{Cr^3}{V_x(r)}\int_{r^2}^{R^2}t^{-3/2}\di t
 \le C\Phi_x(r).
 \end{aligned}
\]
The remaining tail is at most $C\Phi_x(R)\le C\Phi_x(r)$.

For \(R\le r\le3R\), the same small-time calculation gives
$\int_0^{R^2}p_t^Z(P,\zeta)\di t\le C\Phi_x(R)$ uniformly in
\(\zeta\). Together with \eqref{eq:gbl-tail}, this shows that
\(G_P(\zeta)\le C\Phi_x(R)\), and local doubling yields the stronger estimate
\(G_P(\zeta)\le C\Phi_x(r)\). The estimate
\(G_P(\zeta)\le C\Phi_x(R)\) remains valid for all \(r\ge R\);
hence \(G_P\) is finite at every \(\zeta\ne P\).
Since \(\Phi_x(r)\to\infty\) as \(r\downarrow0\), the lower bound
also yields \(G_P(\zeta)\to+\infty\) as \(\zeta\to P\).

\medskip\noindent
\textbf{Continuity and decay at infinity.}
For continuity, let \(K\Subset Z\setminus\{P\}\) and put
\(
 a:=\dist_Z(P,K)>0.
\)
For every \(\zeta\in K\), \eqref{eq:gbl-gaussian} gives
\[
 p_t^Z(P,\zeta)
 \le\frac{C}{V_x(\sqrt t)}e^{-a^2/(Ct)},
 \qquad 0<t\le R^2,
\]
while the product formula and the \(L^\infty\)-bound used in
\eqref{eq:gbl-tail} yield
\[
 p_t^Z(P,\zeta)\le CMt^{-3/2},
 \qquad t>R^2.
\]
For \(0<t\le R^2\), Bishop--Gromov comparison on \(Z\) gives
\[
 \frac1{V_x(\sqrt t)}
 \le\frac{C}{V_x(R)}\left(\frac R{\sqrt t}\right)^{N+3}
 =C_0t^{-(N+3)/2}.
\]
The two preceding bounds are controlled by
\[
 g(t):=
 \begin{cases}
  C_0t^{-(N+3)/2}e^{-a^2/(Ct)}, & 0<t\le R^2,\\[2pt]
  CMt^{-3/2}, & t>R^2.
 \end{cases}
\]
Since \(p_t^Z(P,\zeta)\) is continuous
in \(\zeta\) for each fixed \(t>0\), dominated convergence shows that
\(G_P\) is continuous on \(K\). As \(K\) is arbitrary,
\(G_P\) is continuous on \(Z\setminus\{P\}\)
\cite[Chapter~1]{Rudin87}.

For every fixed \(T\ge R^2\), the Gaussian bound on \((0,T]\)
(with constants allowed to depend on \(T/R^2\), as recalled in
Section~\ref{subsec:ms-background}) and dominated convergence give
\[
 \int_0^T p_t^Z(P,\zeta)\di t\to0
 \quad\text{as }\dist_Z(P,\zeta)\to\infty.
\]
Indeed, for \(\dist_Z(P,\zeta)\ge R\), the Gaussian factor at
small times is bounded by \(e^{-R^2/(C_Tt)}\), which makes the
Gaussian upper bound integrable at \(t=0\); away from \(t=0\)
the heat kernel is uniformly bounded. By \eqref{eq:gbl-tail},
\[
 0\le\limsup_{\dist_Z(P,\zeta)\to\infty}G_P(\zeta)
 \le CMT^{-1/2}.
\]
Letting \(T\to\infty\) proves
\[
 G_P(\zeta)\to0
 \quad\text{as }\dist_Z(P,\zeta)\to\infty.
\]
For every $a>0$, the set $\{G_P\ge a\}$ is closed, by continuity away
from $P$ and the assigned value $G_P(P)=+\infty$, and bounded, by the
decay at infinity; it is therefore compact, since $Z$ is proper.

\medskip\noindent
\textbf{Regularization and harmonicity.}
For $T>\eps$, let
$G_{\eps,T}:=\int_\eps^T p_t^Z(P,\cdot)\di t$.
The function $f:=p_{\eps/2}^Z(P,\cdot)$ belongs to $L^2(Z)$, since
$\|f\|_2^2=p_\eps^Z(P,P)<\infty$. We use the semigroup identity
\[
 \begin{aligned}
 \int_a^b P_s^Zg\di s&\in D(\Delta_Z),\qquad
 -\Delta_Z\int_a^b P_s^Zg\di s
 =P_a^Zg-P_b^Zg,
 \qquad 0\le a<b,\quad g\in L^2(Z),
 \end{aligned}
\]
Hence 
\[
 G_{\eps,T}=\int_{\eps/2}^{T-\eps/2}P_s^Zf\di s
 \in D(\Delta_Z)
\]
and
\[
 -\Delta_ZG_{\eps,T}
 =P_{\eps/2}^Zf-P_{T-\eps/2}^Zf
 =p_\eps^Z(P,\cdot)-p_T^Z(P,\cdot).
\]
By \eqref{eq:gbl-tail}, $G_{\eps,T}\to G_\eps$ uniformly as
$T\to\infty$. Fix \(U\Subset U'\Subset Z\) and choose a
Lipschitz cutoff \(\chi\) satisfying \(0\le\chi\le1\), \(\chi=1\) on
\(U\), and \(\supp\chi\subset U'\). 
For \(T_1,T_2>\eps\),
\[
 \Delta_Z(G_{\eps,T_1}-G_{\eps,T_2})
 =p_{T_1}^Z(P,\cdot)-p_{T_2}^Z(P,\cdot).
\]
The Caccioppoli estimate \eqref{eq:irc-caccioppoli} gives
\[
 \begin{aligned}
 \int\chi^2|D(G_{\eps,T_1}-G_{\eps,T_2})|^2\di\nu
 &\le C\int\chi^2|p_{T_1}^Z(P,\cdot)-p_{T_2}^Z(P,\cdot)|
        |G_{\eps,T_1}-G_{\eps,T_2}|\di\nu\\
 &\quad+C\int|D\chi|^2(G_{\eps,T_1}-G_{\eps,T_2})^2\di\nu.
 \end{aligned}
\]
Both \(G_{\eps,T_1}-G_{\eps,T_2}\) and
\(p_{T_1}^Z(P,\cdot)-p_{T_2}^Z(P,\cdot)\) tend uniformly to zero as
$T_1,T_2\to\infty$. The right-hand side above therefore tends to zero,
so \((G_{\eps,T})_{T>\eps}\) is Cauchy in \(W^{1,2}(U)\). Its uniform
limit is \(G_\eps\); passing to the limit in the weak equation proves
\eqref{eq:gbl-regularized}.

Next, to study the limit as \(\eps\downarrow0\), take
\(U\Subset U'\Subset Z\setminus\{P\}\) and set
\(A:=\dist_Z(P,U')>0\).
Since \(G_P-G_\eps=\int_0^\eps p_t^Z(P,\cdot)\di t\), the same
estimates also give
\[
 0\le G_P(\zeta)-G_\eps(\zeta)
 \le C\int_0^\eps t^{-(N+3)/2}e^{-A^2/(Ct)}\di t
 \to 0
\]
uniformly on \(U'\). Hence both \(G_\eps-G_\eta\) and its Laplacian
\(p_\eta^Z(P,\cdot)-p_\eps^Z(P,\cdot)\) tend uniformly to zero on
\(U'\) as \(\eps,\eta\downarrow0\).

The same Caccioppoli estimate, with the fixed cutoff \(\chi\), therefore
shows that \((G_\eps)_{\eps>0}\) is Cauchy in \(W^{1,2}(U)\). Finally,
for \(\varphi\in\Lip_c(U)\), the weak equation gives
\[
 \int_Z\langle DG_\eps,D\varphi\rangle\di\nu
 =\int_Z p_\eps^Z(P,\cdot)\varphi\di\nu.
\]
Letting \(\eps\downarrow0\) yields
\[
 \int_Z\langle DG_P,D\varphi\rangle\di\nu=0.
\]
Thus \(G_P\) is weakly harmonic on \(U\). Since \(U\) is arbitrary,
\(G_P\) is harmonic on \(Z\setminus\{P\}\).

\medskip\noindent
\textbf{Local gradient bound.}
As \(G_P\) is positive and weakly
harmonic there, the Cheng--Yau estimate for positive harmonic
functions on \(\RCD\) spaces \cite[Theorem~1.2]{HuaKellXia16} applies.
For \(r=\dist_Z(P,\zeta)>0\), set \(\rho:=r/4\). The ball
\(B_{2\rho}^Z(\zeta)\) lies in \(Z\setminus\{P\}\), and therefore
\[
 \frac{|DG_P|(\zeta)}{G_P(\zeta)}
 \le C(N+3)\left(\frac1\rho+\sqrt{K_-}\right)
 \le C(N+3)\left(\frac1r+\sqrt{K_-}\right)
\]
for \(\nu\)-almost every \(\zeta\).
If \(0<r\le2R\), then
\[
 \frac1r+\sqrt{K_-}
 =\frac{1+\sqrt{K_-}\,r}{r}
 \le\frac{1+2\sqrt{K_-R^2}}r.
\]
Together with \(G_P(\zeta)\le C\Phi_x(r)\) from
\eqref{eq:gbl-green-bounds}, this gives
\[
 |DG_P|(\zeta)
 \le C(N,K_-R^2)\frac{G_P(\zeta)}r
 \le C\frac{\Phi_x(r)}r.
\]
In particular, \(|DG_P|\) is locally bounded on
\(Z\setminus\{P\}\). The local Sobolev-to-Lipschitz property therefore
gives a locally Lipschitz representative of \(G_P\) on
\(Z\setminus\{P\}\).
\end{proof}

\subsection{Two auxiliary lemmas}
\label{subsec:gbl-bands}
\begin{lemma}
\label{lem:gbl-flux}
For every $0<a<b<\infty$,
\begin{equation}
 \int_{\{a<G_P<b\}}|DG_P|^2\di\nu=b-a.
 \label{eq:gbl-flux}
\end{equation}
\end{lemma}

\begin{proof}
Set \(T_{a,b}(t):=(t-a)_+\wedge(b-a)\),
Since \(G_P\to\infty\) at \(P\), \(T_{a,b}(G_P)\) is constant in a
neighborhood of \(P\); since \(\{G_P\ge a\}\) is compact, $T_{a,b}(G_P)$ has
compact support. Moreover, \(G_P\) is locally \(W^{1,2}\) and
continuous on \(Z\setminus\{P\}\), so
\[
 T_{a,b}(G_P)\in C_c(Z)\cap W_c^{1,2}(Z),
 \qquad
 DT_{a,b}(G_P)=\mathbf1_{\{a<G_P<b\}}DG_P.
\]
Testing the weak equation \eqref{eq:gbl-regularized} gives
\[
 \int_Z\langle DG_\eps,DT_{a,b}(G_P)\rangle\di\nu
 =\int_Z p_\eps^Z(P,\cdot)T_{a,b}(G_P)\di\nu.
\]
By the strong
\(W^{1,2}_{\loc}\)-convergence of \(G_\eps\) to \(G_P\), its limit is
\[
 \int_Z\langle DG_P,DT_{a,b}(G_P)\rangle\di\nu
 =\int_{\{a<G_P<b\}}|DG_P|^2\di\nu.
\]
On the right, the measures \(p_\eps^Z(P,\cdot)\nu\) converge weakly to
\(\delta_P\), and \(T_{a,b}(G_P)\) is continuous with compact support; hence
\[
 \int_Z p_\eps^Z(P,\cdot)T_{a,b}(G_P)\di\nu\to b-a.
\]
This completes the proof.
\end{proof}
\begin{lemma}
\label{lem:gbl-ordered}
Let $0<a<b\le c<d$ and $K_a:=\overline{\{G_P>a\}}$.
If $w\in W^{1,2}_{\loc}(U)$ is weakly subharmonic on an open
neighborhood $U$ of the compact set $K_a$, then
\begin{equation}
 \frac1{d-c}\int_{\{c<G_P<d\}}w\,|DG_P|^2\di\nu
 \le
 \frac1{b-a}\int_{\{a<G_P<b\}}w\,|DG_P|^2\di\nu.
 \label{eq:gbl-ordered}
\end{equation}
\end{lemma}

\begin{proof}
The set $K_a$ is a closed subset of the compact set $\{G_P\ge a\}$,
so $K_a\Subset U$.
Define
\[
 \Psi(t):=
 \frac{(t-a)_+^2-(t-b)_+^2}{2(b-a)}
 -\frac{(t-c)_+^2-(t-d)_+^2}{2(d-c)},
 \qquad t\ge0.
\]
Then \(\Psi\in C^{1,1}([0,\infty))\), \(\Psi\ge0\),
\(\Psi=0\) on \([0,a]\), and \(\Psi=\Psi(d)\) on \([d,\infty)\), with
\[
 \Psi'(t)=
 \frac{(t-a)_+-(t-b)_+}{b-a}
 -\frac{(t-c)_+-(t-d)_+}{d-c}
\]
and
\[
 \Psi''=\frac{\mathbf1_{(a,b)}}{b-a}
         -\frac{\mathbf1_{(c,d)}}{d-c}
 \quad\text{a.e.}
\]
Since \(G_P\) diverges at \(P\), the function \(\Psi(G_P)\), extended
at \(P\) by \(\Psi(d)\), is constant near \(P\). Its support is
contained in \(K_a\Subset U\), and the chain rule gives
\[
 D\Psi(G_P)=\Psi'(G_P)DG_P.
\]
Also,
\[
 \supp\Psi'(G_P)\subset K_a\cap\{G_P\le d\}
 \Subset U\setminus\{P\},
\]
so the zero extension of \(w\Psi'(G_P)\) is an admissible test
function.
The chain rule gives
\[
 0=\int_U\Psi'(G_P)\langle DG_P,Dw\rangle\di\nu
  +\int_Uw\,\Psi''(G_P)|DG_P|^2\di\nu.
\]
On the other hand, \(\Psi(G_P)\ge0\) and
\(\supp\Psi(G_P)\Subset U\), so the weak subharmonicity of \(w\) gives
\[
 0\le-\int_U\langle Dw,D\Psi(G_P)\rangle\di\nu
   =-\int_U\Psi'(G_P)\langle Dw,DG_P\rangle\di\nu
   =\int_Uw\,\Psi''(G_P)|DG_P|^2\di\nu.
\]
Substituting the displayed formula for \(\Psi''\) yields
\[
 0\le\frac1{b-a}\int_{\{a<G_P<b\}}w|DG_P|^2\di\nu
      -\frac1{d-c}\int_{\{c<G_P<d\}}w|DG_P|^2\di\nu,
\]
which is \eqref{eq:gbl-ordered}.
\end{proof}

\subsection{Proof of Theorem~\ref{thm:gbl-averages}}
\label{subsec:gbl-completion}

\begin{proof}
Let $c,C>0$ be the constants in \eqref{eq:gbl-green-bounds} and fix an
integer $m\ge3$ such that $C2^{-m}\le c/2$; since $c,C$
depend only on $N$ and $K_-R^2$, so does $m$. Set
\[
 s_j:=2^{-mj}\frac{R}{64},
 \qquad
 \alpha_j:=\max_{\partial B_{s_j}^Z(P)}G_P,
 \qquad \beta_j:=\min_{\partial B_{s_{j+1}}^Z(P)}G_P.
\]
Each sphere $\partial B_s^Z(P)$ is nonempty, it contains the
Euclidean sphere $\{x\}\times\{z\in\setR^3:|z|=s\}$, and compact,
since $Z$ is proper. $\alpha_j,\beta_j$ are
attained because $G_P$ is continuous on $Z\setminus\{P\}$. By
\eqref{eq:gbl-green-bounds} and \eqref{eq:gbl-gauge},
\[
\begin{aligned}
 0<\alpha_j\le C\Phi_x(s_j)
 \le C2^{-m}\Phi_x(s_{j+1}),\qquad
 c\Phi_x(s_{j+1})\le\beta_j\le C\Phi_x(s_{j+1}).
\end{aligned}
\]
Since $\beta_j$ and $\alpha_{j+1}$ are the minimum and maximum of the
same function on the same sphere, and $C2^{-m}\le c/2$,
\begin{equation}
 0<\alpha_j<\beta_j\le\alpha_{j+1},
 \qquad
 \frac c2\Phi_x(s_{j+1})
 \le\beta_j-\alpha_j\le C\Phi_x(s_{j+1}).
 \label{eq:gbl-level-choice}
\end{equation}
Define
\[
 \di\sigma_j:=
 \frac{\mathbf1_{\{\alpha_j<G_P<\beta_j\}}|DG_P|^2}
      {\beta_j-\alpha_j}\di\nu.
\]
Lemma~\ref{lem:gbl-flux} gives $\sigma_j(Z)=1$.
The proof of part~\textup{(ii)} only requires
Lemma~\ref{lem:gbl-ordered} and the localization established next.
We prove part~\textup{(ii)} first and complete the projected-density
estimate in part~\textup{(i)} afterwards.

\smallskip\noindent
\emph{Localization in part~\textup{(i)}.}
We first establish the localization needed in part~\textup{(ii)}.
We shall use the following elementary observation on each connected
component of the domains below: if \(h\) is harmonic on a connected
open set and \(h_+\) is compactly supported in it, then \(h\le0\).
Indeed, testing the weak equation for \(h\) with \(h_+\) gives
\[
 0=\int\langle Dh,Dh_+\rangle\di\nu
  =\int_{\{h>0\}}|Dh|^2\di\nu.
\]
Thus \(Dh_+=0\), and \(h_+\) is locally constant. As it has compact
support, \(h_+=0\).

For \(\eta>0\), apply this observation to
\[
 h^-:=G_P-\alpha_j-\eta
 \quad\text{on }Z\setminus\overline B_{s_j}^Z(P).
\]
Since \(G_P\le\alpha_j\) on
\(\partial B_{s_j}^Z(P)\) and \(G_P\to0\) at infinity, the positive
part of \(h^-\) is compactly supported in this domain. Hence \(h^-\le0\).
Letting \(\eta\downarrow0\) and using continuity, we obtain
\[
 G_P\le\alpha_j\quad\text{on }
 Z\setminus B_{s_j}^Z(P).
\]
Likewise, apply the observation to the harmonic function
\[
 h^+:=\beta_j-\eta-G_P
 \quad\text{on }B_{s_{j+1}}^Z(P)\setminus\{P\}.
\]
Its positive part is compactly supported, because \(G_P\ge\beta_j\)
on \(\partial B_{s_{j+1}}^Z(P)\) and \(G_P\to+\infty\) at \(P\).
It follows that \(h^+\le0\); letting \(\eta\downarrow0\) and using
continuity yields
\[
 G_P\ge\beta_j\quad\text{on }
 \overline B_{s_{j+1}}^Z(P).
\]
Consequently,
\[
 E_j:=\{\alpha_j<G_P<\beta_j\}
 \subset\{s_{j+1}<\dist_Z(P,\cdot)<s_j\},
\]
and \(\supp\sigma_j\) lies in the corresponding closed annulus.

\smallskip\noindent
\emph{Proof of \textup{(ii)}}
Let $U$ be an open neighborhood of
$\overline B_{s_j}^Z(P)$, and let
\(w\in W^{1,2}_{\loc}(U)\) be weakly subharmonic in \(U\).
We apply Lemma~\ref{lem:gbl-ordered} with
\[
 a=\alpha_j,\qquad b=\beta_j,\qquad
 c=\alpha_{j+1},\qquad d=\beta_{j+1}.
\]
By \eqref{eq:gbl-level-choice}, for every \(j\),
\[
 0<a<b\le c<d.
\]
Moreover, the exterior localization just proved gives
\[
 K_{\alpha_j}:=\overline{\{G_P>\alpha_j\}}
 \subset\overline B_{s_j}^Z(P)\Subset U.
\]
The remaining assumptions of the lemma are precisely the
regularity and weak subharmonicity of \(w\). Thus
Lemma~\ref{lem:gbl-ordered} applies. In view of the definitions,
its conclusion is exactly
\[
 \int_Zw\di\sigma_{j+1}\le\int_Zw\di\sigma_j.
\]

\smallskip\noindent
\emph{Projected density in part~\textup{(i)}.}
It remains to prove the localized support and the projected
density bound asserted in part~\textup{(i)}. The localization of
the level sets has already been established above.
The gradient estimate \eqref{eq:gbl-green-bounds} and gauge
monotonicity \eqref{eq:gbl-gauge} give
$|DG_P|\le C\Phi_x(s_{j+1})/s_{j+1}$ almost everywhere on \(E_j\).
Together with \eqref{eq:gbl-level-choice} and
$\Phi_x(s)=s^2/V_x(s)$, this yields
\[
\begin{aligned}
\frac{\di\sigma_j}{\di\nu}
 =\frac{\mathbf1_{E_j}|DG_P|^2}{\beta_j-\alpha_j}
 \le C\frac{\Phi_x(s_{j+1})}{s_{j+1}^2}\mathbf1_{E_j}
 \le\frac{C}{V_x(s_{j+1})}\mathbf1_{B_{s_j}^Z(P)}.
\end{aligned}
\]
Thus $\sigma_j$ has bounded density and compact support.
By Tonelli's theorem and the product metric, for
$\meas$-almost every $y$,
\[
\begin{aligned}
 \frac{\di(\pi_X)_\#\sigma_j}{\di\meas}(y)
 =\int_{\setR^3}\frac{\di\sigma_j}{\di\nu}\di z
 \le
 \frac{C\bigl(s_j^2-\dist(x,y)^2\bigr)_+^{3/2}}
      {V_x(s_{j+1})}
 \le
 \frac{Cs_j^3}{V_x(s_{j+1})}\mathbf1_{B_{s_j}(x)}(y).
\end{aligned}
\]
Since $s_j/s_{j+1}=2^m$ is fixed, the volume estimate
\eqref{eq:gbl-volume-local} and relative volume comparison imply
\[
 \frac{s_j^3}{V_x(s_{j+1})}
 \le\frac{C2^{3m}}{\meas(B_{s_{j+1}}(x))}
 \le\frac{C}{\meas(B_{s_j}(x))}.
\]
Hence
\begin{equation}
 \frac{\di(\pi_X)_\#\sigma_j}{\di\meas}(y)
 \le
 \frac{C\mathbf1_{B_{s_j}(x)}(y)}{\meas(B_{s_j}(x))}
 \quad\text{for }\meas\text{-a.e. }y.
 \label{eq:gbl-projection-proof}
\end{equation}
This completes the proof of~\textup{(i)}.

\smallskip\noindent
\emph{Proof of \textup{(iii)}: recovery at Lebesgue points.}
Let $x$ be a Lebesgue point of $F\in L^1_{\loc}(X)$. Since
$\sigma_j(Z)=1$ and $s_j\downarrow0$, \eqref{eq:gbl-projection-proof}
gives
\[
\begin{aligned}
 \left|\int_ZF(y)\di\sigma_j(y,z)-F(x)\right|
 &\le\int_Z|F(y)-F(x)|\di\sigma_j(y,z)\\
 &\le C\fint_{B_{s_j}(x)}|F(y)-F(x)|\di\meas(y)\to0,
\end{aligned}
\]
where the last convergence holds by the Lebesgue differentiation
theorem recalled in Section~\ref{subsec:ms-background}. This
proves~\textup{(iii)} and completes the proof of
Theorem~\ref{thm:gbl-averages}.
\end{proof}

{\setlength{\emergencystretch}{2em}

}

\bigskip
{\small
\noindent\textsc{Bang-Xian Han}\\
School of Mathematics, Shandong University, Jinan 250100, China\\
\textit{Email address:} \href{mailto:hanbx@sdu.edu.cn}{\texttt{hanbx@sdu.edu.cn}}
\par\medskip
\noindent\textsc{Zeyu Wang}\\
School of Mathematics, Shandong University, Jinan 250100, China\\
\textit{Email address:} \href{mailto:202520995@mail.sdu.edu.cn}{\texttt{202520995@mail.sdu.edu.cn}}
\par
}

\end{document}